\documentclass[11pt]{amsart}
\usepackage{amsopn, amssymb, amscd, amsmath, amsthm}
\usepackage{graphicx}
\usepackage{enumerate}
\usepackage{array, multirow}
\newcommand{\nc}{\newcommand}

\nc{\eg}{\mathfrak{e} } \nc{\fg}{\mathfrak{f} } \nc{\qg}{\mathfrak{q} }
\nc{\vg}{\mathfrak{v} } \nc{\wg}{\mathfrak{w} }
\nc{\zg}{\mathfrak{z} } \nc{\ngo}{\mathfrak{n} }
\nc{\kg}{\mathfrak{k} } \nc{\mg}{\mathfrak{m} }
\nc{\bg}{\mathfrak{b} } \nc{\ggo}{\mathfrak{g} }
\nc{\ggob}{\overline{\mathfrak{g}} } \nc{\sog}{\mathfrak{so} }
\nc{\sug}{\mathfrak{su} } \nc{\spg}{\mathfrak{sp} }
\nc{\slg}{\mathfrak{sl} } \nc{\glg}{\mathfrak{gl} }
\nc{\cg}{\mathfrak{c} } \nc{\rg}{\mathfrak{r} }
\nc{\hg}{\mathfrak{h} } \nc{\tg}{\mathfrak{t} }
\nc{\ug}{\mathfrak{u} } \nc{\dg}{\mathfrak{d} }
\nc{\ag}{\mathfrak{a} } \nc{\pg}{\mathfrak{p} }
\nc{\sg}{\mathfrak{s} } \nc{\affg}{\mathfrak{aff} }
\nc{\Xg}{\mathfrak{X} }

\nc{\pca}{\mathcal{P}} \nc{\nca}{\mathcal{N}}
\nc{\lca}{\mathcal{L}} \nc{\oca}{\mathcal{O}}
\nc{\mca}{\mathcal{M}} \nc{\tca}{\mathcal{T}}
\nc{\aca}{\mathcal{A}} \nc{\cca}{\mathcal{C}}
\nc{\gca}{\mathcal{G}} \nc{\sca}{\mathcal{S}}
\nc{\hca}{\mathcal{H}} \nc{\bca}{\mathcal{B}}
\nc{\dca}{\mathcal{D}} \nc{\val}{\operatorname{val}}

\nc{\vp}{\varphi} \nc{\ddt}{\tfrac{{\rm d}}{{\rm d}t}}
\nc{\dpar}{\tfrac{\partial}{\partial t}} \nc{\im}{\mathtt{i}}

\nc{\SO}{\mathrm{SO}} \nc{\Spe}{\mathrm{Sp}} \nc{\Sl}{\mathrm{SL}}
\nc{\SU}{\mathrm{SU}} \nc{\Or}{\mathrm{O}} \nc{\U}{\mathrm{U}}
\nc{\Gl}{\mathrm{GL}} \nc{\Se}{\mathrm{S}} \nc{\Cl}{\mathrm{Cl}}
\nc{\Spein}{\mathrm{Spin}} \nc{\Pin}{\mathrm{Pin}}
\nc{\G}{\mathrm{GL}_n(\RR)} \nc{\g}{\mathfrak{gl}_n(\RR)}

\nc{\RR}{{\mathbb R}} \nc{\HH}{{\mathbb H}} \nc{\CC}{{\mathbb C}}
\nc{\ZZ}{{\mathbb Z}} \nc{\FF}{{\mathbb F}} \nc{\NN}{{\mathbb N}}
\nc{\QQ}{{\mathbb Q}} \nc{\PP}{{\mathbb P}}

\nc{\vs}{\vspace{.2cm}} \nc{\vsp}{\vspace{1cm}}
\nc{\ip}{\langle\cdot,\cdot\rangle} \nc{\ipp}{(\cdot,\cdot)}
\nc{\la}{\langle} \nc{\ra}{\rangle} \nc{\unm}{\tfrac{1}{2}}
\nc{\unc}{\tfrac{1}{4}} \nc{\und}{\tfrac{1}{16}}
\nc{\no}{\vs\noindent} \nc{\lamn}{\Lambda^2(\RR^n)^*\otimes\RR^n}
\nc{\lamp}{\Lambda^2\pg^*\otimes\pg}
\nc{\lamg}{\Lambda^2\ggo^*\otimes\ggo}
\nc{\lamngo}{\Lambda^2\ngo^*\otimes\ngo} \nc{\tangz}{{\rm T}^{\rm
Zar}} \nc{\mum}{/\!\!/} \nc{\kir}{/\!\!/\!\!/}
\nc{\Ri}{\tfrac{4\Ric_{\mu}}{||\mu||^2}} \nc{\ds}{\displaystyle}
\nc{\ben}{\begin{enumerate}} \nc{\een}{\end{enumerate}}
\nc{\f}{\frac} \nc{\lb}{[\cdot,\cdot]}
\nc{\isn}{\tfrac{1}{||v||^2}} \nc{\gkp}{(\ggo=\kg\oplus\pg,\ip)}
\nc{\ukh}{(\ug=\kg\oplus\hg,\ip)} \nc{\gkhn}{\ggo = \kg \oplus \hg \oplus \ngo}
\nc{\mrd}{(\ggo = \kg \oplus \pg, \ip)}

\nc{\Hess}{\operatorname{Hess}} \nc{\ad}{\operatorname{ad}}
\nc{\Ad}{\operatorname{Ad}} \nc{\rank}{\operatorname{rank}}
\nc{\Irr}{\operatorname{Irr}} \nc{\End}{\operatorname{End}}
\nc{\Aut}{\operatorname{Aut}} \nc{\Inn}{\operatorname{Inn}}
\nc{\Der}{\operatorname{Der}} \nc{\Ker}{\operatorname{Ker}}
\nc{\Iso}{\operatorname{I}} \nc{\Diff}{\operatorname{Diff}}
\nc{\Lie}{\operatorname{Lie}} \nc{\tr}{\operatorname{tr}}
\nc{\dif}{\operatorname{d}} \nc{\sen}{\operatorname{sen}}
\nc{\modu}{\operatorname{mod}} \nc{\Riem}{\operatorname{Rm}}
\nc{\Ricci}{\operatorname{Ric}} \nc{\sym}{\operatorname{sym}}
\nc{\symac}{\operatorname{sym^{ac}}}
\nc{\symc}{\operatorname{sym^{c}}} \nc{\scalar}{\operatorname{sc}}
\nc{\grad}{\operatorname{grad}} \nc{\ricci}{\operatorname{Ric}}
\nc{\nr}{\operatorname{nr}} \nc{\riccic}{\operatorname{ric^{c}}}
\nc{\riccig}{\operatorname{ric^{\gamma}}}
\nc{\Rin}{\operatorname{M}} \nc{\Le}{\operatorname{L}}
\nc{\tang}{\operatorname{T}} \nc{\level}{\operatorname{level}}
\nc{\rad}{\operatorname{r}} \nc{\abel}{\operatorname{ab}}
\nc{\CH}{\operatorname{CH}} \nc{\mcc}{\operatorname{mcc}}
\nc{\Adj}{\operatorname{Adj}} \nc{\Order}{\operatorname{O}}
\nc{\inj}{\operatorname{inj}}\nc{\R}{\operatorname{R}}
\nc{\Spec}{\operatorname{Spec}} \nc{\I}{\operatorname{I}}
\nc{\Rc}{\operatorname{Rc}} \nc{\Nu}{\operatorname{Nu}}

\theoremstyle{plain}
\newtheorem{theorem}{Theorem}[section]
\newtheorem{proposition}[theorem]{Proposition}
\newtheorem{corollary}[theorem]{Corollary}
\newtheorem{lemma}[theorem]{Lemma}
\newtheorem{teointro}{Theorem}

\newtheorem*{AC}{Alekseevskii's conjecture}
\newtheorem*{GAC}{Generalized Alekseevskii's conjecture}

\theoremstyle{definition}

\theoremstyle{remark}
\newtheorem{remark}[theorem]{Remark}

\title{Homogeneous Ricci solitons in low dimensions}

\author{Romina M.~Arroyo} \author{Ramiro Lafuente}

\address{FaMAF y CIEM, Universidad Nacional de C\'ordoba, C\'ordoba, Argentina}

\email{arroyo@famaf.unc.edu.ar, rlafuente@famaf.unc.edu.ar}

\thanks{This research was supported by a fellowship from CONICET and grants from CONICET, FonCyT and SeCyT (Universidad Nacional de C\'ordoba)}


\begin{document}

\begin{abstract}
In this article we classify expanding homogeneous Ricci solitons up to dimension $5$, according to their presentation as homogeneous spaces. We obtain that they are all isometric to solvsolitons, and this in particular implies that the generalized Alekseevskii conjecture holds in these dimensions. In addition, we prove that the conjecture holds in dimension $6$ provided the transitive group is not semisimple.
\end{abstract}

\maketitle

%\tableofcontents

\section{Introduction}\label{int}

A complete Riemannian metric $g$ on a differentiable manifold $M$ is a \emph{Ricci soliton} if its Ricci tensor satisfies
\begin{equation}\label{defsoliton}
    \Rc(g) = c g + \lca_X g, \qquad \hbox{for some } c\in \RR, \qquad X\in \Xg(M),
\end{equation}
where $\lca_X g$ denotes the usual Lie derivative of $g$ in the direction of the complete vector field~ $X$. The importance of these metrics relies on the fact that they are precisely the self-similar solutions to the Ricci flow, and so they arise naturally in the study of its singularities. Moreover, they are natural generalizations of the notion of \emph{Einstein metric} (i.e.~$\Rc(g) = c g$).

In the homogeneous case, Einstein and Ricci soliton metrics have been extensively studied by many authors (see for instance the surveys \cite{cruzchica}, \cite{Wng2}, and the references therein). However, its classification is yet not fully understood. This article is devoted to the classification of homogeneous Ricci solitons in low dimensions.

In the \emph{shrinking} case ($c>0$), they are all given by quotients of a product of a compact Einstein homogeneous manifold with a Euclidean space (\cite{PW}). The classification of compact Einstein homogeneous manifolds in low dimensions has been studied in \cite{AleDotFer}, \cite{NikRod}, \cite{Nkn1} and \cite{BhmKrr}, among many others. On the other hand, \emph{steady} homogeneous Ricci solitons ($c=0$) are necessarily flat, since it easily follows that they must be Ricci flat, and so flat by \cite[Theorem 1]{AlkKml}. Thus, we focus on the \emph{expanding} case ($c<0$).

It is well known that in dimensions $2$ and $3$, Einstein metrics are necessarily of constant sectional curvature. Simply connected Einstein homogeneous $4$-manifolds have been classified in \cite{Jns}, where it is shown that they are all isometric to symmetric spaces. The classification of $5$-dimensional simply connected Einstein homogeneous manifolds with negative Ricci curvature has been carried out in \cite{Alek75} in the case of non-positive sectional curvature, and in \cite{Nkn1} in the general case.

Regarding expanding homogeneous Ricci solitons, it is worth pointing out that up to now all known examples are isometric to a \emph{solvsoliton}, that is, a left-invariant metric on a simply connected solvable Lie group $S$ whose (1,1)-Ricci tensor at $\sg = \Lie(S)$ satisfies
\begin{equation}\label{defsolvsoliton}
    \ricci = c I + D, \qquad c\in \RR, \qquad D\in \Der(\sg).
\end{equation}
When the group is nilpotent, they are called \emph{nilsolitons}. It was shown in \cite{finding} and \cite{Wll03} that any nilpotent Lie group up to dimension $6$ admits a nilsoliton metric (which is unique up to isometry and scaling by a result of Heber, see \cite{Heb}). Nilsolitons in dimension $7$ have been recently classified in \cite{FC13}. Solvsolitons of dimension up to $4$ have been classified in \cite{solvsolitons}, and using the structural results from that article a classification up to dimension $6$ was given in \cite{Wll}.

For the general homogeneous case, recent results from \cite{Jbl}, \cite{alek} and \cite{Jbl13b} imply that any homogeneous Ricci soliton can be presented as a Riemannian homogeneous space $(G/K,g)$ which is an \emph{algebraic soliton}, i.e.~ its (1,1)-Ricci tensor at $T_{eK} G/K$ satisfies
\begin{equation}\label{defalgsoliton}
    \ricci = c I + D_\pg, \qquad c\in \RR, \qquad D\in \Der(\ggo),
\end{equation}
for some reductive decomposition $\ggo = \kg \oplus \pg$, where $D_\pg$ denotes the restriction of $D$ to $\pg$ (see Section~ \ref{pre} for more details about algebraic solitons).

Summarizing, in order to classify homogeneous Ricci solitons, we will restrict ourselves to study expanding algebraic solitons. As we have already mentioned, Einstein metrics of negative Ricci curvature and solvsolitons have been previously classified up to dimension $5$, so we will omit them.

Our first main result is the classification of simply connected expanding algebraic solitons of dimensions $3$ and $4$ up to equivariant isometry (that is, with respect to the transitive group~ $G$; see Section \ref{prehomog} for a precise definition). 

%Explicar bien donde esta la clasificacion de Einstein hasta dim 5, y en que terminos es dicha clasificacion (isometria? equiv isom?)

\begin{teointro}\label{main34}
Any simply connected expanding algebraic soliton of dimension $3$ or $4$ is either Einstein, or equivariantly isometric to one of the algebraic solitons in Tables \ref{ta3} and \ref{ta4}.
\end{teointro}

\begin{table}
{\center
\begin{tabular}{|c|>{\centering}m{3cm}|c|>{\centering}m{8cm}|c|}
  \hline
 %  after \\: \hline or \cline{col1-col2} \cline{col3-col4} ...
 $\dim G$ & $G$ & $K$ & Metric & $G/K$ \\
  \hline
  $3$ & $S$ solvable Lie group & $e$ & Solvsoliton. & $S$ \\
  \hline
  \multirow{2}{*}{$4$} & $\SO(2) \ltimes S$ & $\SO(2)$ & Solvsoliton (see Corollary \ref{gsol}). & $S$ \\\cline{2-5}
    & $\Sl_2(\RR) \times \RR$ & $\SO(2)$ & Product of metrics of constant curvature on $\Sl_2(\RR)/\SO(2)$ and $\RR$. & ${\RR H}^2 \times \RR$ \\
  \hline
\end{tabular}}

\vs \caption{Non-Einstein simply connected expanding algebraic solitons of dimension $3$, up to equivariant isometry.}\label{ta3}
\end{table}

% \begin{table}
% {\center
% \begin{tabular}{|c|c|c|c|c|}
%   \hline
%  %  after \\: \hline or \cline{col1-col2} \cline{col3-col4} ...
%  $\dim G$ & $G$ & $K$ & Metric & $G/K$ \\
%   \hline
%   $3$ & $S$ solvable Lie  & $e$ & See the classification of solvsolitons in & $S$ \\
%   & group & & \cite[Table 1]{solvsolitons} and nilsolitons in \cite[Table 1]{Wll}.&\\
%   \hline
%   $4$ & $\SO(2) \ltimes S$ & $\SO(2)$ &See Corollary \ref{gsol} & $S$ \\\cline{2-5}
%     & $\Sl_2(\RR) \times \RR$ & $\SO(2)$ & Product of metrics of constant curvature  & ${\RR H}^2 \times \RR$ \\
%    & & & on $\Sl_2(\RR)/\SO(2)$ and $\RR$. &\\
%   \hline
% \end{tabular}}

% \vs \caption{Non-Einstein simply connected expanding algebraic solitons of dimension $3$, up to equivariant isometry.}\label{ta3}
% \end{table}

%See the classification of solvsolitons in \cite[Table 2]{solvsolitons} and nilsolitons in \cite[Table 2]{Wll}.

\begin{table}
{\center
\begin{tabular}{|c|c|c|>{\centering}m{4.7cm}|c|}
  \hline
  $\dim G$ & $G$ & $K$ & Metric & $G/K$\\
  \hline
  $4$    &    $S$ solvable Lie group &  $e$ & Solvsoliton. &  $S$ \\
  \hline
  \multirow{3}{*}{$5$} & $\SO(2) \ltimes S$ & $\SO(2)$ & Solvsoliton (see Corollary~ \ref{gsol}). & $S$ \\ \cline{2-5}
      & $\Sl_2(\RR) \times \RR^2$ & $\SO(2)$ & Product of metrics of constant curvature on $\Sl_2(\RR)/\SO(2)$ and  $\RR^2.$ & $\RR H^2 \times \RR^2$ \\ \cline{2-5}
      & $\Sl_2(\RR) \ltimes \RR^2$  & $\SO(2)$ & Non-product metrics where the curvature takes both signs (see Section \ref{M4G5N2}). & $S$ \\
   \hline
  $6$ %&                                                  &                       & Product of a Einstein metric & $S^2\times \RR^2$ \\
      %& $\SO(3) \times \left(\SO(2)\ltimes \RR^2\right)$ & $\SO(2)\times \SO(2)$ & on $\SO(3)/\SO(2)$ and a flat & shrinking. \\
      %&                                                  &                       & metric on $\left(\SO(2)\ltimes \RR^2\right)/\SO(2).$ &\\ \cline{2-5}
  %    &                                                     &                       &    &\\
      & $\Sl_2(\RR)\times \left(\SO(2)\ltimes \RR^2\right)$ & $\SO(2)\times \SO(2)$ & Product of metrics of constant curvature on $\Sl_2(\RR)/\SO(2)$ and $\left(\SO(2)\ltimes \RR^2\right)/\SO(2)$.  & ${\RR H}^2 \times \RR^2$  \\
   \hline
  $7$ %& $\RR \times \SO(4)$ & $\SO(3)$ & Product of a metric of $\RR$ and a &  $S^3\times \RR$ \\
      %&                     &          & metric of positive constant curvature. & shri  &                        &          &                              & \\
     & $\RR \times \SO_0(3,1)$ & $\SO(3)$ &  Product of metrics of constant curvature on $\SO_0(3,1)/\SO(3)$ and $\RR$. & ${\RR H}^3\times \RR$ \\
     %&                         &          &   &\\
  \hline
\end{tabular}}
\vs \caption{Non-Einstein simply connected expanding algebraic solitons of dimension $4$, up to equivariant isometry.}\label{ta4}
\end{table}

We now make a few remarks concerning Tables \ref{ta3}, \ref{ta4} and \ref{ta5}. The presentation and notation is according to the classification of Riemannian homogeneous spaces of dimension $3$ and $4$ given in \cite{BB4}. In particular, all homogeneous spaces $G/K$ in those tables are assumed to be effective. The right column gives the underlying manifold (they are all diffeomoprhic to $\RR^n$), and $S$ always denotes a simply connected solvable Lie group of the corresponding dimension. Finally, in Table \ref{ta5}, $H_3$ denotes the $3$-dimensional Heisenberg group, and we refer the reader to Section \ref{m5g6n3} for a clarification of the notation $\theta_{\ad}$ and $\theta_{12}$.

The main tools used for the classification are the structural results for algebraic solitons given in \cite{alek}. Roughly speaking, it is shown there that for a simply connected algebraic soliton $(G/K,g)$ with $G$ simply connected, one has that $G$ splits as $G\simeq U \ltimes N$ with $N$ its nilradical and $U$ a reductive Lie group, $G/K = U/K \times N$ as differentiable manifolds, $(N,g_N)$ is a nilsoliton, and some additional compatibility conditions concerning $(U/K, g_{U/K}), (N,g_N)$ and the action of $U$ on $N$ must be satisfied. Here, $g_{U/K}$ and $ g_N$ are the induced metrics. See Theorem \ref{solisemal} for a more precise statement.

With the additional hypothesis that the homogeneous space $U/K$ is almost effective (which can always be assumed by suitably shrinking the transitive group $G$; see Corollary \ref{U/Kae}), a classification of $5$-dimensional simply connected expanding algebraic solitons up to equivariant isometry can be obtained. In order to simplify the presentation, we will also omit here the cases of \emph{trivial} solitons (product of an Einstein homogeneous manifold with a Euclidean space). 
%
%if and only if there exists a Riemannian submersion $(G/K,g) \to (U/K,g|_{U/K})$ with fibers isometric to $(N,g_N)$, where $N$ is the connected nilradical of $G$, $g_N$ is a nilsoliton metric, $U \simeq G/N$ is a reductive Lie group, and some additional properties regarding the fundamental O'Neill tensors of the submersion $A$ and $T$ are satisfied

\begin{teointro}\label{main5}
Any simply connected expanding algebraic soliton of dimension $5$ with the property that $U/K$ is almost effective is either the product of a non-compact Einstein homogeneous manifold with a Euclidean space, or is equivariantly isometric to one of the algebraic solitons in Table \ref{ta5}.
\end{teointro}

\begin{table}
{\center
\begin{tabular}{|c|c|c|>{\centering}m{7cm}|c|}
  \hline
  $\dim G$ & $G$ & $K$ & Metric & $G/K$\\
  \hline
  $5$ & $S$ solvable Lie group & $e$ & Solvsoliton. &  $S$ \\

  \hline
  \multirow{4}{*}{$6$} & $\Sl_2(\RR) \times S$ & $\SO(2)$ & Product of an Einstien metric on $\Sl_2(\RR)/\SO(2)$ and a $3$D solvsoliton with the same cosmological constants. & $\RR H^2 \times S$  \\ \cline{2-5}
    % & $\U(1,1) \ltimes \RR^2$ & $\U(1)$ & . & $\widetilde{\Sl_2(\RR)} \times \RR^2$ \\ \cline{2-5}
     & $\Sl_2(\RR) \ltimes_{\theta_{\ad}} \RR^3$ &  $\SO(2)$ & Determined by $\ip_{\alpha,\beta,\beta}$ on $T_{eK}G/K$ (see Section~ \ref{m5g6n3}). & $S$ \\ \cline{2-5}
     & $\Sl_2(\RR) \ltimes_{\theta_{12}} \RR^3$ &  $\SO(2)$ & Any $G$-invariant metric. & $S \times \RR$ \\ \cline{2-5}
%     & $\Sl_2(\RR) \times H_3$ &  $\SO(2)$ & Product of an Einstein metric and the nilsoliton $H_3$. & $\RR H^2 \times H_3$ \\ \cline{2-5}
     & $\Sl_2(\RR) \ltimes_{\theta_{12}} H_3$ &  $\SO(2)$ & Non-product, determined by $\ip_{{4\beta^2/\gamma},\beta,\gamma}$ on $T_{eK}G/K$ (see Section~ \ref{m5g6n3}). & $S$ \\
  \hline
\end{tabular}
}
\vs \caption{\emph{Non-trivial} simply connected expanding algebraic solitons of dimension $5$ with $U/K$ almost effective, up to equivariant isometry.}\label{ta5}
\end{table}

One of the most important open problems on Einstein homogeneous manifolds is the following (see \cite[7.57]{Bss}):
\begin{AC}
Any connected homogeneous Einstein manifold of negative scalar curvature is diffeomorphic to a Euclidean space.
\end{AC}
Recently, it was proved in \cite{alek} (and also in \cite{HePtrWyl} by a different approach) that the conjecture is actually equivalent to the following analogous statement for expanding algebraic solitons:
\begin{GAC}
Any expanding algebraic soliton is diffeomorphic to a Euclidean space.
\end{GAC}

The proofs of Theorems \ref{main34} and \ref{main5} together with the results on the Einstein case from \cite{Jns}, \cite{Nkn1}, and the reduction to the simply connected case given in \cite[Theorem 1.1]{Jbl13c}, immediatly imply the following:

\begin{corollary}\label{cormain}
Any simply connected expanding algebraic soliton of dimension up to $5$ is isometric to a solvsoliton. In particular, the generalized Alekseevskii conjecture holds up to dimension~ $5$.
\end{corollary}

\begin{remark}
Notice that the above statement for the ``generalized Alekseevskii conjecture'' is precisely the one given in \cite[\S 6]{alek} (though it was not given that name there). This is slightly different from the stronger version stated in \cite{Jbl13c}, where the conclusion ``diffeomorphic to a Euclidean space'' has been replaced by ``isometric to a solvsoliton''. It is yet unkonwn if this later statement is equivalent to the original Alekseevskii conjecture.
\end{remark}

Finally, we verify that the generalized Alekseevskii conjecture holds for simply connected spaces of dimension $6$, except for the cases where there is a transitive semisimple group. As far as we know, this was not previously known even in the case of Einstein metrics and the original Alekseevskii conjecture.

\begin{teointro}\label{main6}
If $(G/K,g)$ is a simply connected expanding algebraic soliton of dimension $6$ and $G$ is not semisimple, then $G/K$ is diffeomorphic to $\RR^n$.
\end{teointro}

The organization of the paper will be as follows: In Section \ref{pre} we will review some basic facts about Riemannian homogeneous spaces, and we will also briefly introduce the reader to the theory of homogeneous Ricci solitons. The main results from \cite{alek} on the algebraic structure of algebraic solitons will be also presented there, together with some new applications. The proof of Theorem \ref{main34} will be given in Sections \ref{dim3} and \ref{dim4}. Theorem \ref{main5} will be proved in Section \ref{dim5}, and Theorem \ref{main6} in Section \ref{dim6}. Finally, at the end there is an appendix kindly written by Jorge Lauret, in which some additional results regarding the structure of algebraic solitons will be given.

\vs \noindent {\it Acknowledgements.} It is our pleasure to thank Jorge Lauret for very fruitful discussions, for his useful comments on a first version of this paper, and for writing the Appendix.

\section{Preliminaries}\label{pre}

\subsection{Riemannian homogeneous spaces}\label{prehomog}$\,$

We review in this section some well-known facts about homogeneous spaces with invariant Riemannian metrics. For a more detailed exposition on this matter, we refer the reader to \cite[\S 2]{alek}. See also \cite{BB4}, where there is a similar treatment of these topics (notice, however, that the notation there is slightly different from ours).

A \emph{Riemannian homogeneous space} $(G/K,g)$ is a differentiable manifold $G/K$ where $G$ is a Lie group and $K\subseteq G$ a closed subgroup, endowed with a $G$-invariant Riemannian metric (i.e.~a metric such that the natural action of $G$ on the coset space $G/K$ is by isometries). In this paper, we will assume that all manifolds are connected, and that $G$ is connected as well, unless otherwise stated. We will also assume that the Lie subgroup of the full isometry group $\Iso(G/K,g)$ corresponding to $G$ is closed.

Two Riemannian homogeneous spaces $(G_1/K_1,g_1)$ and $(G_2/K_2,g_2)$ are said to be \emph{equivariantly isometric} if there exists a Lie group isomorphism $\Phi: G_1 \to G_2$ with $\Phi(K_1) = K_2$ such that the induced diffeomorphism $\varphi : G_1/K_1 \to G_2/K_2$ is an isometry. This is the most natural equivalence relation between Riemannian homogeneous spaces, and it is of course stronger than Riemannian isometry.

Given a Riemannian homogeneous space $(G/K,g)$, if $\ggo = \Lie(G)$, $\kg = \Lie(K) \subseteq \ggo$, we let $\pg$ be the orthogonal complement of $\kg$ with respect to the Killing form of $\ggo$ (which is negative definite on $\kg$). Then, $\pg$ is a reductive complement for $\kg$, there is a natural identification $\pg \simeq T_{eK} G/K$, and $\ip = g(eK)$ is an inner product on $\pg$ which is $\Ad(K)$-invariant. We call the data set $(\ggo = \kg \oplus \pg, \ip)$ a \emph{metric reductive decomposition}. The action of $G$ on $G/K$ is \emph{almost effective} if the isotropy representation $\ad : \kg \to \End(\pg)$ is faithful, and it is effective if $\Ad : K \to \Gl(\pg)$ is faithful (in this last case, $K$ turns out to be compact). Recall that if $G/K$ is simply connected then $K$ is necessarily connected.

Conversely, to a given metric reductive decomposition we can associate a simply connected Riemannian homogeneous space, which is uniquely determined up to equivariant isometry if we assume either that $G$ is simply connected, or that the action is effective. Indeed, given $(\ggo = \kg \oplus \pg, \ip)$, one may take $G$ the simply connected Lie group with Lie algebra $\ggo$, $K$ the connected Lie subgroup of $G$ corresponding to $\kg$ (which will be assumed to be closed), and then on the manifold $G/K$ we define a Riemannian metric $g$ which is given by left translation by elements of $G$ of the inner product $\ip$ defined on $\pg \simeq T_{eK} G/K$.

There are some restrictions for the dimension of $G$ in terms of the dimension of the manifold, given by the following classical result.

\begin{proposition}\cite{Wang47}\label{gap}
If $(M^n,g)$ is a Riemannian manifold with $n \neq 4$, then the isometry group $\Iso(M,g)$ does not contain any closed subgroup of dimension $d$ with
\begin{equation}\label{wang}
\tfrac{n(n-1)}{2}+1 < d < \tfrac{n(n+1)}{2}.
\end{equation}
\end{proposition}

Recall that if $\dim \Iso(M,g)=\tfrac{n(n+1)}{2}$ then $(M,g)$ has constant sectional curvature (see \cite[1.79]{Bss}).

\subsection{Homogeneous Ricci solitons}$\,$

Recall that a complete Riemannian manifold $(M,g)$ is called a \emph{Ricci soliton} if its Ricci tensor satisfies \eqref{defsoliton}. The constant $c$ is often called the \emph{cosmological constant}. Ricci solitons are natural generalizations of the concept of Einstein metrics. In addition, they are precisely the self-similar solutions to the Ricci flow. Indeed, $g$ is a Ricci soliton if and only if the one-parameter family of metrics
\[
  g(t) = (-2ct+1) \varphi^*_t g
\]
is a solution to the Ricci flow, where $\varphi_t$ is (up to time reparameterization) the one-parameter family of diffeomorphisms of $M$ generated by $X$.

When $(M,g)$ is homogeneous, we can present it as a Riemannian homogeneous space $(G/K,g)$, and the fact that the isometry group is preserved by the Ricci flow (see \cite{Kot}) allows us to use the same presentation group $G$ for all metrics $g(t)$. If $X$ is such that $\varphi_t$ are $G$-equivariant diffeomorphisms (by this we mean that they are defined by Lie automorphisms of $G$ that leave $K$ invariant), then $(G/K,g)$ is called a \emph{semi-algebraic soliton}. It was proved in \cite{Jbl} that every homogeneous Ricci soliton is a semi-algebraic soliton with respect to its full isometry group~ $\Iso(M,g)$.

It is also shown in \cite{Jbl} that if we consider a reductive decomposition $(\ggo = \kg \oplus \pg,\ip)$ for a semi-algebraic soliton $(G/K,g)$, then its Ricci operator (the $(1,1)$-Ricci tensor, at the point~ $eK$) is given by
\[
  \Ricci(g) = c I + S(D_\pg)\in \End(\pg), \qquad c\in \RR,
\]
where $D_\pg$ is the restriction to $\pg$ of a derivation $D\in \Der(\ggo)$ with $D\kg = 0$, and $S(A) = \unm (A + A^t)$.

Recently, it was proved that every semi-algebraic soliton is actually \emph{algebraic} (see \cite{alek} for the unimodular case, and \cite{Jbl13b} for the non-unimodular case), meaning that the derivation $D$ in the previous formula may be chosen to be symmetric, so that $\Ricci(g) = c I + D_\pg$. Thus, in order to classify homogeneous Ricci solitons we may restrict ourselves to classify algebraic solitons. This is important because algebraic solitons satisfy much nicer structural properties beyond the fact of $D$ being symmetric.

The concept of algebraic soliton generalizes that of a \emph{solvsoliton}, i.e.~a simply connected solvable Lie groups $S$ endowed with a left-invariant metric $g$ that satisfies $\Ricci(g)=c I + D$, with $D\in \Der(\sg)$, $\sg=\Lie(S)$ (they are called \emph{nilsolitons} in the nilpotente case). Up to know, all known-examples of expanding homogeneous Ricci solitons are isometric to a solvsoliton.

From results due to Ivey, Naber and Petersen-Wylie (see \cite{Iv93}, \cite{Nab10}, \cite{PW}), any \emph{non-trivial} (that is, non Einstein nor the product of an Einstein homogeneous manifold with a Euclidean space) Ricci soliton must necesarily be non-compact, \emph{expanding} ($c<0$) and non-gradient.
 %This is why we will focus on classifying expanding algebraic solitons.

See \cite{homRS} and \cite{Jbl} for a more detailed exposition on homogeneous Ricci solitons.

\subsection{Algebraic structure of homogeneous Ricci solitons and its consequences}$\,$

In this section, we will briefly recall the main structural results for homogeneous Ricci solitons given in \cite{alek}, since they will be constantly used along this paper.

Let $(G/K,g)$ be an almost effective Riemannian homogeneous space endowed with a metric reductive decomposition $\mrd$. Let us further decompose $\pg$ as $\pg = \hg \oplus \ngo$, where $\ngo$ is the nilradical of $\ggo$ (which is contained in $\pg$ because it is Killing-orthogonal to $\kg$), and $\la \hg,\ngo \ra = 0$. Then, $\ggo$ decomposes as follows:
\begin{equation}\label{descdeg}
    \ggo= \rlap{$\overbrace{\phantom{\kg\oplus\hg}}^\ug$} \kg \oplus \underbrace{\hg\oplus\ngo}_\pg,
\end{equation}
where $\ug = \kg \oplus \hg$ is just a vector subspace, and the Lie bracket satisfies the following relations (see \cite[\S 2]{alek}):
\begin{equation}\label{relengkhn}
[\kg,\kg] \subset \kg, \quad [\kg, \hg] \subset \hg, \quad [\ggo, \ngo] \subset \ngo.
\end{equation}
The next theorem is the main result of \cite{alek}, and gives necessary and sufficient conditions for the above data in order to define an expanding algebraic soliton.
%This result will be used along the paper.

\begin{theorem}\cite[Theorem 4.6]{alek}\label{solisemal}
Let $(G/K,g)$ be a Riemannian homogeneous space endowed with a metric reductive decomposition $(\ggo=\kg \oplus \pg, \ip)$ such that $B(\kg,\pg)=0$, where $B$ is the Killing form of $\ggo$, and consider in $\ggo$ the decomposition (\ref{descdeg}). If $(G/K,g)$ is an algebraic soliton with cosmological constant $c<0$, then the following conditions hold:
\begin{itemize}
\item[(i)] $[\hg,\hg] \subset \kg \oplus \hg.$ In particular, $\ug= \kg \oplus \hg$ is a reductive Lie subalgebra of $\ggo$ (i.e. $\ug$ is the direct sum of a semisimple ideal and its center), and $\ggo=\ug \ltimes \ngo$.
\item[(ii)] $\ricci_{\ug}=cI+C_{\theta},$ where $\ricci_{\ug}$ is the Ricci operator of the metric reductive decomposition $(\ug= \kg \oplus \hg,\ip|_{\hg \times \hg})$, and $C_{\theta}$ is the symmetric operator defined by
    \[
    \la C_{\theta}Y,Y\ra=\tr(S(\ad Y|_{\ngo})^2), \quad \forall Y \in \hg.
    \]
    Here, $\theta: \ug \to \Der(\ngo)$ is the Lie algebra representation of $\ug$ given by $\theta(Y) = \ad Y|_\ngo$.
\item[(iii)] $\ricci_{\ngo}=cI+D_1,$ for some $D_1 \in \Der(\ngo),$ where $\ricci_{\ngo}$ denotes the Ricci operator of the nilpotent metric Lie algebra $(\ngo, \ip|_{\ngo \times \ngo})$ (i.e. $(\ngo, \ip|_{\ngo \times \ngo})$ is a nilsoliton).
\item[(iv)] $\sum [\ad Y_i|_{\ngo}, (\ad Y_i|_{\ngo})^t]=0,$ where $\{Y_i\}$ is any orthonormal basis of $\hg$ (in particular, $(\ad Y|_{\ngo})^t \in \Der(\ngo)$ for all $Y \in \hg$).
\item[(v)] The Ricci operator of $(G/K,g)$ is given by $\ricci = c I + D_{\pg},$ where
$$
D:= -S(\ad H) + \left(
               \begin{array}{ccc}
                 0 &  &  \\
                  & 0 &  \\
                  &  & D_1 \\
               \end{array}
             \right) \in \Der(\ggo),
$$ and $H \in \ggo$ is the unique element that satisfies $\la H,X \ra = \tr \ad X,$ for all $X \in \ggo.$
\end{itemize}
Conversely, if conditions (i)-(iv) hold and $G/K$ is simply connected, then $(G/K,g)$ is an algebraic soliton with cosmological constant $c$ and derivation $D$ as above.
\end{theorem}
%\begin{corollary}\cite[Corollary 4.11]{alek}\label{colosemial}
%Let $(G/K,g)$ be a algebraic soliton with cosmological constant $c<0$ and with a metric reductive decomposition $(\ggo=\kg \oplus \pg, \ip)$ such that $B(\kg,\pg)=0.$ Consider in $\ggo$ the decomposition (\ref{descdeg}). If $[\hg,\hg]=0,$ then $(G/K,g)$ is isometric to a solvmanifold. This in particular holds when $\ggo$ is solvable.
%\end{corolla

\begin{remark}
Theorem 4.6 of \cite{alek} is actually stated for semi-algebraic solitons. The above statement for algebraic solitons only differs on the formulas for $\ricci$ and $D$ in condition (v).
\end{remark}

We now make some observations regarding the group-level decomposition. Let $U, N$ be the connected Lie subgroups of $G$ with Lie algebras $\ug$ and $\ngo$, respectively. The subgroup $N$ is the \emph{nilradical} of $G$, i.e.~the maximal connected nilpotent normal subgroup of $G$. Now assume that $G/K$ is simply connected. Then, $K$ is connected and $K\subseteq U$. Let us call $g_{U/K}$ the metric on $U/K$ induced by $g$. The data set $(\ug= \kg \oplus \hg,\ip|_{\hg \times \hg})$ is a metric reductive decomposition for $(U/K,g_{U/K})$ (though it may happen that $\hg$ is not $B_\ug$-orthogonal to $\kg$, where $B_\ug$ is the Killing form of $\ug$), and $\ricci_\ug$ is its corresponding Ricci operator. Observe that the homogeneous space $U/K$ may not be almost-effective; however, we may always assume that it is so by suitably shrinking the transitive group $G$, as it will be shown in Corollary \ref{U/Kae}.

\begin{remark}\label{Gsc}
If $G$ is simply connected, then $G \simeq U \ltimes N,$ and if $G/K$ is also simply connected, we have $G/K=U/K \times N$ as differentiable manifolds.
\end{remark}

\begin{remark}\label{theta0}
The fact that $\hg \perp \ngo$ allows us to recover the homogeneous space $(G/K,g)$ from the data $(N,g_N)$, $(U/K,g_{U/K})$ and the action of $U$ on $N$ induced by $\theta$. Moreover, if $\theta = 0$, then $(G/K,g) = (U/K,g|_{U/K}) \times (N,g|_N)$ is clearly a Riemannian product. If in addition $(N,g_N)$ is flat, then the algebraic soliton is \emph{trivial}.
\end{remark}

\begin{proposition}\label{algebraic}
Let $(M,g)$ be a homogeneous Ricci soliton which is algebraic with respect to a transitive group $G$, and let $D$ be the corresponding derivation of $\ggo := \Lie(G)$. Let $G_1$ be a Lie subgroup of $G$ which is also transitive on $M$, $\ggo_1 = \Lie(G_1)$, and assume that $D(\ggo_1) \subseteq \ggo_1$. Then, $(M,g)$ is also algebraic with respect to $G_1$. Moreover, a sufficient condition in order for $D(\ggo_1)\subseteq \ggo_1$ to hold is that $G_1$ contains the nilradical of $G$.
\end{proposition}
\begin{proof}
The first claim follows from \cite[Proposition A.1]{HePtrWyl} and its proof. Indeed, using that $(M,g)$ is $G$-algebraic, it can be shown that the soliton vector field $X$ can be chosen so that $\lca_X$ acts on $\ggo$ as the algebraic soliton derivation $D$, under the natural identification of vector fields with elements of $\ggo$.

The last claim is a consequence of Theorem \ref{solisemal}, (v) and \cite[Proposition 4.14]{alek}, which show that for an algebraic soliton, the image of the corresponding derivation is contained in the nilradical.
\end{proof}

The following will be very important for proving Theorems \ref{main5} and \ref{main6}.

\begin{corollary}\label{U/Kae}
Let $(G/K,g)$ be a simply connected, expanding algebraic soliton. Then, there exists a connected Lie subgroup $G_1$ of $G$ which is transitive on $G/K$, with isotropy $K_1 = G_1 \cap K$, such that $(G_1/K_1,g)$ is still an algebraic soliton, and such that the corresponding homogeneous space $U_1/K_1$ is almost effective.
\end{corollary}
\begin{proof}
Let $(\gkhn,\ip)$ be a metric reductive decomposition for $(G/K,g)$, and consider the ideal of $\ug$ given by $\qg = \ker \left\{ \ad : \kg \to \End(\hg)\right\}\subseteq \kg$. Since $\ug$ is reductive, there exists an ideal $\ug_1\subseteq \ug$ such that $\ug = \qg \oplus \ug_1$. Let $\ggo_1 := \ug_1\oplus \ngo$, which is clearly a subalgebra of $\ggo$ (it is in fact an ideal), and let $G_1$ be the connected Lie subgroup of $G$ with Lie algebra $\ggo_1$. Then $G_1$ is still transitive on $G/K$ (because $K G_1 = G$), and we have a presentation $(G_1/K_1,g)$, with $K_1 = G_1\cap K$, which by Proposition \ref{algebraic} is still an algebraic soliton. Finally, if $U_1$ is the connected Lie subgroup of $G_1$ with Lie algebra $\ug_1$, the homogeneous space $U_1/K_1$ is almost-effective by construction.
%By Theorem \ref{solisemal}, any expanding algebraic soliton is isometric to an algebraic soliton constructed from a metric reductive decomposition $(\gkhn,\ip)$ that satisfies (i)-(iv). If the homogeneous space associated to the metric reductive decomposition $(\ug = \kg \oplus \hg, \ip|_\hg)$ is not almost effective, consider the ideal of $\kg$ given by $\qg = \ker \left\{ \ad : \kg \to \End(\hg)\right\}$, and consider an ideal $\tilde{\ug}$ of $\ug$ $\tilde{\kg}:= \kg / \qg$. The homogeneous space associated to $(\tilde{\ug} = \tilde{\kg}\oplus \hg, \ip|_\hg)$ is isometric to the previous one, but this is almost effective. Finally, it is clear that $(\tilde{\ggo} = \tilde{\kg}\oplus \hg \oplus \ngo, \ip)$ satisfies (i)-(iv) of Theorem \ref{solisemal}, and the algebraic soliton obtained from this data is isometric to the one considered at the beginning of the proof.
\end{proof}

An important application of Theorem \ref{solisemal} is the following proposition, which reduces to the unimodular case the classification of homogeneous Ricci solitons up to isometry.

\begin{proposition}\cite[\S 6]{alek}\label{unimod}
Let $(G/K,g)$ be a simply connected expanding algebraic soliton, and assume that $G$ is non-unimodular. Then, there exists a codimension-one normal subgroup $G_1 \vartriangleleft G$ which is unimodular, contains $K$, and the homogeneous space $(G_1/K, g|_{G_1/K})$ is a simply connected, expanding algebraic soliton.
\end{proposition}

%The decomposition \eqref{descdeg} and the conditions from Theorem \ref{solisemal} will be constantly used along this paper.
%Let us call $\theta:\ug \rightarrow \Der(\ngo)$ the Lie algebra representation induced by the adjoint representation of $\ggo$, that is,
%\begin{equation}\label{dtheta}
%\theta(X)=\ad(X)|_{\ngo}, \quad X \in \ug.
%\end{equation}

In order to simplify the presentation, we summarize in the following lemma some border cases which can be easily handled.
\begin{lemma}\label{lemadimn}
Let $(G/K,g)$ be a Riemannian homogeneous space which is an algebraic soliton with cosmological constant $c<0$ and consider the decomposition $\gkhn$ as in \eqref{descdeg}.
\begin{enumerate}[(i)]
  \item\label{lemadimn-n0}\cite{alek} If $\ngo=0$ (i.e. $\ggo$ semisimple), then $(G/K,g)$ is Einstein with $\ricci = cI$.
  \item\label{lemadimn-n1uss} If $\dim \ngo=1$ and $\ggo$ is unimodular, then the representation $\theta:\ug \rightarrow \Der(\ngo)$ is trivial. This applies in particular when $\ug$ is semisimple and $\dim \ngo = 1$.
   \item\label{lemadimn-n1zh} If $\dim \ngo=1,$ then $\zg(\ug) \subset \hg,$ where $\zg(\ug)$ is the center of the Lie algebra $\ug.$
   \item\label{lemadimn-zn2}  $\dim \zg(\ug)\leq (\dim \ngo)^2.$
   \item\label{lemadimn-u2} If $\dim \ug \leq 2$ then $G$ is solvable and $(G/K,g)$ is isometric to a solvsoliton. If in addition we have that $\kg=0$, then $(G,g)$ is a solvsoliton.
\end{enumerate}
\end{lemma}

\begin{proof} By Theorem \ref{solisemal}, (ii), we have that $C_{\theta}=0$ and thus $\ricci_{\ug}=cI$. Then $(G/K,g)$ is Einstein with $\ricci=cI$, since $U = G$ in this case, and (i) follows.

If $\ggo$ is unimodular then $\tr \theta(X) = 0$, $\forall X \in \ug$, since $\ug$ is also unimodular. This implies that $\theta = 0$ if $\dim \ngo = 1$. When $\ug$ is semisimple, $\theta(\ug) = [\theta(\ug),\theta(\ug)]$ and thus $\theta = 0$ if $\dim \ngo = 1$.

Let $Z \in \kg$ and $W \in \zg(\ug)$. The inner product $\ip$ on $\pg$ is $\Ad(K)$-invariant, hence $\kg$ acts on $\ngo$ by skew-symmetric endomorphisms. Therefore, $\theta(Z)=0$ if $\dim \ngo=1$. This implies that $Z$ and $W$ are $B$-orthogonal, so (iii) follows from the fact that $\pg$ is the $B$-orthogonal complement of $\kg$.

If $\dim \zg(\ug)> (\dim \ngo)^2 = \dim \End(\ngo)$, then there exists a nonzero $Z\in \zg(\ug)\cap \ker (\theta) \subseteq \zg(\ggo) \subseteq \ngo$. This is a contradiction, therefore (iv) is proved.

The assumption $\dim \ug \leq 2$ implies that $\ug$ is abelian and hence $\ggo$ is solvable. Thus $(G/K,g)$ is a Ricci soliton solvmanifold, which by \cite[Theorem 1.1]{Jbl} is isometric to a solvsoliton. If $\dim \kg = 0$ then $g$ is a left-invariant Ricci soliton metric on the solvable Lie group $G$, that is also an algebraic soliton, so by definition it is a solvsoliton.
\end{proof}

\section{Proof of Theorem \ref{main34}: $\dim G/K = 3$}\label{dim3}

In this section, we classify $3$-dimensional expanding algebraic solitons up to equivariant isometry. By Proposition \ref{gap}, the possible values for $\dim \ggo$ are $3,4$ and $6$. Moreover, $\dim \ggo=6$ if and only if $(G/K,g)$ is ${\RR H}^3=\SO_0(3,1)/\SO(3)$, which is the only simply connected $3$-dimensional space form of negative curvature, and is of course Einstein. On the other hand, by Lemma~ \ref{lemadimn},~ \eqref{lemadimn-n0}, we know that if $\dim \ngo=0$ then $(G/K,g)$ is an Einstein manifold, thus this cases will be omitted.

We begin with a general result which allows us to deal with the case of solvmanifolds.

\begin{lemma}\label{gsol}
Let $(G/K,g)$ be a simply connected, effective Riemannian homogeneous space which is an expanding algebraic soliton, consider the decomposition $\gkhn$ as in $\eqref{descdeg}$, and assume that $[\hg,\hg] = 0$ (this in particular holds when $G$ is solvable). Then, $(G/K,g)$ is equivariantly isometric to $\left( (K\ltimes S)/K, g_S\right)$, where $S$ is a solvable normal Lie subgroup of $G$ which acts simply transitively on $G/K$. The metric $g_S$ is determined by a solvsoliton inner product $\ip_S$ on $\sg = \Lie(S)$, and $(G/K,g)$ is isometric to the solvsoliton $(S,\ip_S$).
\end{lemma}

\begin{proof}
We apply Theorem \ref{solisemal} to this situation. First observe that if $G$ is solvable then $\ug$ is abelian, since it is reductive and also a subalgebra of a solvable Lie algebra (thus solvable itself). Note that $[\hg,\hg] = 0$ implies that $\pg =\hg \oplus \ngo $ is a solvable ideal of $\ggo$, and let $S$ be the connected Lie subgroup of $G$ with Lie algebra $\pg$. Since $S$ is normal, $KS$ is a subgroup of $G$. Moreover, it is a well known fact that the application $\ggo = \kg \oplus \pg \to G$ given by $(Z,X) \mapsto \exp(Z) \exp(X)$ is a local diffeomorphism. This implies that $KS$ contains a neighborhood of the identity, and therefore $G = KS$. Since $K$ is the isotropy at some point, $S$ must act transitively on $G/K$. The corresponding isotropy is the discrete subgroup $K\cap S$, which is connected because $G/K$ is simply connected (see Section \ref{pre}), therefore it is trivial. We conclude that $S$ is also simply connected, and that the action is simply transitive. Since $K\cap S = e$, it follows that $G \simeq K \ltimes S$.

Let us call $g_S$ the left-invariant metric defined on $S$ by the inner product $g(e)$ on $T_e S \simeq \pg$. It is clear that the metrics $g_S$ and $g$ are equivariantly isometric. Regarding the Ricci operator, from Theorem \ref{solisemal}, (ii) we see that $\ricci_\ug = 0$ so $C_\theta = -cI$. Thus Theorem \ref{solisemal} implies that $(\pg= \hg \oplus \ngo, g_S(e))$, satisfies conditions (i)-(iv) from \cite[Theorem 4.8]{solvsolitons}, from which it follows that $(S,g_S)$ is a solvsoliton, and the proof is concluded.
\end{proof}

\begin{remark}
The fact that $(G/K,g)$ is isometric to a solvsoliton was previously obtained in \cite[Theorem 1.1]{Jbl} for $G$ solvable, and in \cite[Corollary 4.11]{alek} under the hypothesis of the previous lemma.
\end{remark}

Now we proceed with the proof, considering different cases according to $\dim \ggo$ and $\dim \ngo$.

\subsection{$\dim \ggo=3$}
Since we are assuming $\dim \ngo > 0$, it follows from Lemma \ref{lemadimn}, (v) that in this case, $(G,g)$ is always a solvsoliton.

\subsection{$\dim \ggo=4$} We divide into cases according to the dimension of $\ngo$:

\subsubsection{$\dim \ggo=4, \dim \ngo = 1$}$\,$

We have that $\dim \hg=2$, $\dim \ug = 3$. It follows from Theorem \ref{solisemal}, (i) that $\ug$ is a reductive Lie algebra, hence it must be isomorphic to one of the following:
\[
\RR^3, \qquad \slg_2(\RR), \qquad \sug(2).
\]
By Lemma \ref{lemadimn}, \eqref{lemadimn-zn2}, $\ug \neq \RR^3$. Also, $\ug \neq \sug(2)$ by Corollary \ref{Unocomp}. If $\ug \simeq \slg_2(\RR)$ then $\ug$ is semisimple, so by Lemma \ref{lemadimn}, \eqref{lemadimn-n1uss} and Remark \ref{theta0} we obtain that the metric is a product $G/K=U/K \times \RR$. In addition, $\ricci_{U/K}=\ricci_{\ug}=cI$ by (ii) of Theorem \ref{solisemal}, hence $U/K$ is isometric to ${\RR H}^2$, the $2$-dimensional hyperbolic space.

\subsubsection{$\dim \ggo=4, \dim \ngo = 2, 3$}$\,$

In this case, $\ug$ is abelian and $\ggo$ is solvable by Lemma \ref{lemadimn}, (v). Therefore, by Lemma \ref{gsol}, $G/K=(\SO(2)\ltimes S)/\SO(2)$, where $S$ is the connected Lie subgroup of $G$ with Lie algebra $\pg$, and the metric is determined by a solvsoliton inner product on $T_{eK}(G/K) \simeq \pg.$

%\vskip16pt
%
%
%The results obtained in this section can be summarized in the following theorem.
%
%\begin{theorem}\label{toeGK3}
%Any simply connected expanding algebraic soliton of dimension $3$ is Einstein, or is equivariantly isometric to one of the algebraic solitons in Table \ref{ta3}.
%\end{theorem}

%
%\begin{corollary}\label{corGK3}
%Any simply connected expanding algebraic soliton of dimension $3$ is isometric to a solvsoliton.
%\end{corollary}
%\begin{proof}
%Any simply connected $3$-dimensional Einstein manifold with negative scalar curvature is isometric to ${\RR H}^3$ (see \cite[Proposition 1.120]{Bss}), which is clearly a solvsoliton. For the homogeneous spaces in Table \ref{ta3}, the case $(\SO(2) \ltimes S)/\SO(2)$ follows from Lemma \ref{gsol}, and $(\Sl_2(\RR) \times \RR)/\SO(2) \simeq \RR H^2 \times \RR$ is a product of a solvsoliton and a flat factor, which is again a solvsoliton.
%\end{proof}

\section{Proof of Theorem \ref{main34}: $\dim G/K =4$}\label{dim4}

We now focus on the classification of $4$-dimensional, simply connected, expanding algebraic solitons up to equivariant isometry. Our classification is according to \cite[Theorem 4.1]{BB4}, where simply connected, effective Riemannian homogeneous spaces of dimension $4$ are classified up to equivariant isometry. The possible values for $\dim \ggo$ are $4,5,6,7,8$ and $10$. We note that when $\dim \ggo\geq 6$ all homogeneous metrics are either Einstein, or a product of metrics of constant sectional curvature. If one of the factors is flat, one gets a \emph{trivial} soliton (i.e.~ the product of an Einstein manifold with a flat space). We disregard the cases where none of the factors are flat because of the following lemma, which is easy to prove.

\begin{lemma}\label{prodcurvcte}
Let $(M_1,g_1)$ and $(M_2,g_2)$ be two non-flat, Einstein homogeneous Riemannian manifolds with Einstein constants $c_1$ and $c_2$, respectively. Then, $(M_1 \times M_2, g_1 \times g_2)$ is a Ricci soliton if and only if $c_1=c_2$, and this is equivalent to $(M_1 \times M_2, g_1 \times g_2)$ being Einstein.
\end{lemma}

Thus, we are left dealing with the cases $\dim \ggo=4, 5$. Recall that if $\dim \ngo=0$, Lemma~ \ref{lemadimn},~ \eqref{lemadimn-n0} implies that $(G/K,g)$ is Einstein.

\subsection{$\dim \ggo=4$}\label{GK4dimG4}

\subsubsection{$\dim \ggo=4, \dim \ngo = 1$}$\,$

We have that $\dim \hg=3,$ $\ug = \hg$ and since it is reductive, it must be isomorphic to one of the following Lie algebras:
\[
\RR^3, \qquad \slg_2(\RR), \qquad \sug(2).
\]
However, from Lemma \ref{lemadimn}, \eqref{lemadimn-zn2} and Corollaries \ref{Unocomp} and \ref{unosl2} we have that $\hg$ can not be isomorphic to any of the previously mentioned Lie algebras.
%From Lemma \ref{lemadimn}, \eqref{lemadimn-zn2} we have that $\ug \neq \RR^3$. Then $\ug$ must be semisimple, and using Lemma \ref{lemadimn}, \eqref{lemadimn-n1uss} we see that $\theta = 0$. Therefore, by Theorem \ref{solisemal}, (ii) $(U,g)$ is Einstein with $c<0$, and this is a contradiction for both cases: if $\ug = \slg_2(\RR)$, then it is a well-known fact that $U$ does not admit an Einstein left-invariant metric; and if $\ug = \sug(2)$, then $U$ is compact and the contradiction follows from a classic result due to Bochner.
%% and Propositions ......... and .............. ($\hg$ no puede ser $\slg_2(\RR)$ porque la descomposicion de cartan no puede ser ortogonal y $\ug$ no puede ser compacto) we have that $\hg$ can not be isomorphic to $\RR^3,$ $\slg_2(\RR)$ or $\sug(2)$.

\subsubsection{$\dim \ggo = 4, \dim \ngo \geq 2$}$\,$

It follows from Lemma \ref{lemadimn}, \eqref{lemadimn-u2} that in this case $(G,g)$ is a solvsoliton.

\subsection{$\dim \ggo=5$}\label{GK4dimG5}

\subsubsection{$\dim \ggo = 5, \dim \ngo = 1$}$\,$

In this case we have  $\dim \ug=4, \dim \hg=3$. Also, $\ug$ is isomorphic to one of the following Lie algebras:
\[
\RR^4, \qquad \RR \oplus \slg_2(\RR), \qquad \RR \oplus \sug(2).
\]
From Lemma \ref{lemadimn}, (iv) we know that $\ug$ can not be isomorphic to $\RR^4.$ If $\ug \simeq \RR \oplus \slg_2(\RR)$ or $\RR \oplus \sug(2)$ then $\dim \zg(\ug) = 1$, and by Lemma \ref{lemadimn}, \eqref{lemadimn-n1zh} we have that $\zg(\ug) \subset \hg$. Let $\{Z,Y_1,Y_2,Y_3\}$ be an orthonormal basis of $\ug$ such that $Z \in \kg,$ $\{Y_1,Y_2,Y_3\} \in \hg$ and $Y_3 \in \zg(\ug)$. The Lie bracket according to this basis is as follows:
%As $\kg$ leaves $\ug$ invariant, the decomposition is reductive and $e_4 \in \zg(\ug),$ we have that the Lie bracket in this basis is as follows:
\[
[Z,Y_1]= -a Y_2, \quad [Z,Y_2]= a Y_1, \quad [Y_1,Y_2]= b Z + d Y_3, \quad a,b,d \in \RR, \quad a,b \neq 0.
\]
Then, the Ricci operator of $U/K$ on $\{Y_1,Y_2,Y_3\}$ has the following form:
\begin{equation}\label{riccidim4g5n1}
\ricci_{\ug}=\left[
             \begin{array}{ccc}
               -\tfrac{1}{2}d^2 -ba &  &  \\
                & -\tfrac{1}{2}d^2 -ba &  \\
                &  &  \tfrac{1}{2}d^2 \\
             \end{array}
           \right].
\end{equation}
On the other hand, we have that $-a \theta (Y_2)=[\theta (Z), \theta (Y_1)]=0$, so $\theta (Y_2)=0$ and analogously $\theta (Y_1)=0$. Also, $\theta(Z)=0$ because it is skew-symmetric and $\dim \ngo = 1$. Then, $d \theta(Y_3) = 0,$  from which it is follows that either $\theta (Y_3)=0$ or $d=0.$

If $\theta (Y_3)=0$ then $\dim \ngo=2$, a contradiction. Therefore, $\theta (Y_3)\neq0$ and $d=0.$ In this case, from the Ricci operator (\ref{riccidim4g5n1}) and Theorem \ref{solisemal}, (ii) we obtain that the cosmological constant is given by $c=-a b,$ and so $a b>0$ since $c<0.$ It follows that $\ug$ is isomorphic to $\RR \oplus \slg_{2}(\RR),$ and by Theorem \ref{solisemal}, (i) $\ggo \simeq (\RR \oplus \slg_{2}(\RR)) \ltimes \RR.$ Hence, by using \cite[Theorem 4.1]{BB4} we conclude that $G$ is the forth group of the first table (\cite[pp.~44]{BB4}). In particular, the metric $g$ is a product of constant curvature metrics, so by Lemma \ref{prodcurvcte} $(G/K,g)$ is an Einstein manifold.

\subsubsection{$\dim \ggo = 5, \dim \ngo = 2$}\label{M4G5N2}$\,$

Here, $\dim \hg=2$, $\dim \ug=3$, $\ngo = \RR^2$ is abelian, and $\ug$ is isomorphic to one of the following Lie algebras:
\[
\RR^3, \qquad \slg_2(\RR), \qquad \sug(2).
\]
From Corollary \ref{Unocomp} we have that $\ug \neq \sug(2)$. If $\ug = \RR^3$ then by Lemma \ref{gsol}, $(G/K,g)$ is equivariantly isometric to $\left(\SO(2)\ltimes S)/\SO(2), g_S\right)$, where $S$ is the connected Lie subgroup of $G$ with Lie algebra $\pg$ (which is solvable) and $g_S$ is the $\left(\SO(2)\ltimes S\right)$-invariant metric determined by the inner product $g(eK)$ on $T_{eK}(G/K) \simeq \pg.$

If $\ug=\slg_2(\RR)$, then $\theta$ is either trivial or the standard representation of $\slg_2(\RR)$ on $\RR^2$. If it is trivial, then we are in a product case. Otherwise, $G = \Sl_2(\RR) \ltimes \RR^2$. Let $\{Z,Y_1,Y_2\}$ be a basis of $\ug$ such that $Z \in \kg,$ $\{Y_1,Y_2\}\in \hg$ and the Lie bracket is as follows:
\[
[Z,Y_1]=2 Y_2, \quad [Z,Y_2]=-2 Y_1, \quad [Y_1,Y_2]=-2 Z.
\]
Let $\{X_1,X_2\}$ be a basis of $\ngo$ such that, with respect to this basis, $\theta$ is given by
\[
\theta(Z) = \left[ \begin{smallmatrix} 0 & -1 \\ 1 & 0 \end{smallmatrix} \right], \quad
\theta(Y_1) = \left[ \begin{smallmatrix} 1 & 0 \\ 0 & -1 \end{smallmatrix} \right], \quad
\theta(Y_2) = \left[ \begin{smallmatrix} 0 & 1 \\ 1 & 0 \end{smallmatrix} \right], \quad
\]
% i.e.~the seventh group of the first table of \cite[Theorem~ 4.1]{BB4}
%\theta(e_1)=\left(
%              \begin{array}{cc}
%                0 & -1 \\
%                1 & 0 \\
%              \end{array}
%            \right), \quad \theta(e_2)= \left(
%                                          \begin{array}{cc}
%                                            1 & 0 \\
%                                            0 & -1 \\
%                                          \end{array}
%                                        \right), \quad \theta(e_3)=\left(
%                                                         \begin{array}{cc}
%                                                           0 & 1 \\
%                                                           1 & 0 \\
%                                                         \end{array}
%                                                       \right).
and let $\ip_0$ be the inner product on $\pg$ such that $\{Y_1,Y_2,X_1,X_2\}$ is an orthonormal basis. The isotropy representation has two non-equivalent irreducible summands $\hg$ and $\ngo$, thus all $\ad(\kg)$-invariant inner products on $\pg$ are given by
\[
\ip_{\alpha,\beta} = \alpha \cdot \ip_0|_{\hg} + \beta \cdot \ip_0 |_{\ngo}, \qquad \alpha,\beta>0.
\]
An orthonormal basis of $\pg$ with respect to $\ip_{\alpha,\beta}$ is given by $\left\{\frac{Y_1}{\sqrt{\alpha}}, \frac{Y_2}{\sqrt{\alpha}}, \frac{X_1}{\sqrt{\beta}}, \frac{X_2}{\sqrt{\beta}}\right\}$. A straightforward calculation shows that with these inner products $\ip_{\alpha,\beta}$, conditions (i)-(iv) from Theorem~ \ref{solisemal} are satisfied, and thus they all give rise to expanding algebraic solitons.
%Then, the matriz $\ad(e_1)|_{\pg}$ with respect to this basis is
%$$
% \ad(e_1)|_{\pg}=\left(
%  \begin{array}{cccc}
%    0 & -2 &  &  \\
%    2 & 0 &  &  \\
%     &  & 0 & -1 \\
%     &  & 1 & 0 \\
%  \end{array}
%\right).
%$$
%So, the isotropy representation of $\kg$ on $\pg$ has two non-equivalent irreducible components. Let $h_{a,b}=\left[
%                                                                             \begin{array}{cc}
%                                                                               a I_{2 \times 2} &  \\
%                                                                                & b I_{2 \times 2} \\
%                                                                             \end{array}
%                                                                           \right],
%$ with $a>0, b>0.$ Consider the inner products
%\begin{equation}\label{piab}
%\ip_{a,b}:= \la h_{a,b}\cdot, h_{a,b}\cdot \ra.
%\end{equation}
%Then,
%$$
%\begin{array}{lll}
%  \la e_i,e_j \ra_{a,b} & = & a^2 \delta_{ij}, \quad i,j=2,3, \\
%  \la X_k,X_l \ra_{a,b} & = & b^2 \delta_{kl}, \quad k,l=1,2.
%\end{array}
%$$ So, a orthonormal basis of $\pg$ with respect to $\ip_{a,b}$ is $\{\frac{e_2}{a},\frac{e_3}{a}, \frac{X_1}{b},\frac{X_2}{b}\}.$

Observe that, with these metrics, $(\Sl_2(\RR) \ltimes \RR^2)/\SO(2)$ is not isometric to ${\RR H}^2 \times \RR^2$. Indeed, the plane $\pi=\left\la \frac{X_1}{\sqrt{\beta}},\frac{X_2}{\sqrt{\beta}} \right\ra$ has positive sectional curvature $2/\alpha$, and in $\RR H^2 \times \RR^2$ the sectional curvature is everywhere non-positive.

%$$
%\begin{array}{lll}
%K\left(\frac{X_1}{b},\frac{X_2}{b}\right)&=&-\tfrac{3}{4}\left\|\left[\frac{X_1}{b},\frac{X_2}{b}\right]_{\pg}\right\|^2-\tfrac{1}{2}\left\langle \left[\frac{X_1}{b},\left[\frac{X_1}{b},\frac{X_2}{b}\right]_{\ggo}\right]_{\pg},\frac{X_2}{b}\right\rangle \\ & &-\tfrac{1}{2}\left\langle\left[\frac{X_2}{b},\left[\frac{X_2}{b},\frac{X_1}{b}\right]_{\ggo}\right]_{\pg},\frac{X_1}{b} \right\rangle
%+ \left\|U\left(\frac{X_1}{b},\frac{X_2}{b}\right)\right\|^2 \\
%& &-\left\langle U\left(\frac{X_1}{b},\frac{X_1}{b}\right),U\left(\frac{X_2}{b},\frac{X_2}{b}\right) \right\rangle=\frac{1}{a^2}>0.
%\end{array}
%$$ Then $\pi$ has positive sectional curvature and ${\RR H}^2 \times \RR^2$ has every theirs sectional curvatures less than or equal to zero.

%We know that $\ricci_{\ug}=cI+C_{\theta}$ by Theorem \ref{solisemal} (ii). So, we have that
%$$
%\left[
%  \begin{array}{cc}
%    -\frac{1}{a^2} & 0 \\
%    0 & -\frac{1}{a^2} \\
%  \end{array}
%\right]=\ricci_{\ug}=cI+C_{\theta}= c I + \left[
%                                         \begin{array}{cc}
%                                           \frac{2}{a^2} & 0 \\
%                                           0 & \frac{2}{b^2} \\
%                                         \end{array}
%                                       \right].
%$$ It follows that $a=b.$

Finally, we consider $\sg$ the solvable Lie algebra of $\ggo$ generated by $\left\{\tfrac{Y_1}{\sqrt{\alpha}},\tfrac{Z + Y_2}{\sqrt{\alpha}}, \tfrac{X_1}{\sqrt{\beta}},\tfrac{X_2}{\sqrt{\beta}}\right\}.$ The Lie bracket is the following:
\begin{eqnarray}\label{subals}
\left[\tfrac{Y_1}{\sqrt{\alpha}}, \tfrac{Z + Y_2}{\sqrt{\alpha}}\right]= -\tfrac{2}{\sqrt{\alpha}} \left(\tfrac{Z + Y_2}{\sqrt{\alpha}}\right), \quad \left[\tfrac{Y_1}{\sqrt{\alpha}},\tfrac{X_1}{\sqrt{\beta}}\right]=\tfrac{1}{\sqrt{\alpha}}\left(\tfrac{X_1}{\sqrt{\beta}}\right), \\
\nonumber \left[\tfrac{Y_1}{\sqrt{\alpha}}, \tfrac{X_2}{\sqrt{\beta}}\right]=-\tfrac{1}{\sqrt{\alpha}}\left(\tfrac{X_2}{\sqrt{\beta}}\right), \quad \left[\tfrac{Z + Y_2}{\sqrt{\alpha}},\tfrac{X_1}{\sqrt{\beta}}\right]=\tfrac{2}{\sqrt{\alpha}} \left(\tfrac{X_2}{\sqrt{\beta}}\right).
\end{eqnarray}
Let $S$ be the connected Lie subgroup of $G$ with Lie algebra $\sg.$ We observe that $S$ acts transitively on $(\Sl_2(\RR) \ltimes \RR^2)/\SO(2).$ Indeed, if we take an Iwasawa decomposition for $\ug\simeq \slg_2(\RR)$ we get $\ug=\kg \oplus \ag \oplus \widetilde{\ngo}= \la Z \ra \oplus \la Y_1 \ra \oplus \la Z+ Y_2\ra.$ Then, $G= KA\widetilde{N} \ltimes \RR^2,$ where $A$ and $\widetilde{N}$ are the connected Lie subgroups of $U$ with Lie algebras $\ag$ and $\widetilde{\ngo}$ respectively. Observe that the compact factor of the Iwasawa decomposition is precisely the isotropy $K$. Then, it is easy to see that $G=SK$, and it follows that $S$ acts transitively on $G/K$. Thus, $(G/K,g)$ is isometric to $(S,g_S)$, where $g_S$ is the left-invariant metric on $S$ defined by the inner product $\ip_{\alpha, \beta}$ on $T_{eK}G/K = T_e S$. This is a solvsoliton, since its Ricci operator is given by
\[
\ricci=-\tfrac{6}{\alpha}I+\left[
  \begin{smallmatrix}
    0 &  &  &  \\
     & 0 &  &  \\
     &  & \tfrac{6}{\alpha} &  \\
     &  &  & \tfrac{6}{\alpha} \\
  \end{smallmatrix}
\right]=-\tfrac{6}{\alpha}I+D,
\]
and $D \in \Der(\sg)$.

\subsubsection{$\dim \ggo = 5, \dim \ngo \geq 3$}$\,$

In this case, $\ug$ is necessarily abelian and Lemma \ref{gsol} applies.

\section{Proof of Theorem \ref{main5}: $\dim G/K = 5$}\label{dim5}

Our aim in this section is to classify $5$-dimensional simply connected, expanding algebraic solitons up to isometry. But in fact we will do much more than that: we classify them according to their presentation as homogeneous spaces (that is, up to equivariant isometry), with the additional hypothesis that $(U/K, g|_{U/K})$ is almost effective (recall that by Corollary \ref{U/Kae} this condition can always be obtained if we suitably shrink the transitive group $G$). Also, a main goal is to prove that they are all isometric to solvsolitons.

By almost-effectiveness, $\dim U$ is bounded in terms of $\dim U/K$, and this yields
\begin{equation}\label{B}
\dim G \leq \dim N + \tfrac12 (\dim U/K)(\dim U/K + 1).
\end{equation}

As in the previous sections, we divide into cases according to the dimension of $G$ and its nilradical $N$. Recall that if $\dim \ngo = 0$ then by Lemma \ref{lemadimn}, \eqref{lemadimn-n0} $(G/K,g)$ is Einstein. The possible dimensions for $G$ according to Proposition \ref{gap} are $5\leq \dim G \leq 11$ or $\dim G = 15$. The cases with $\dim G = 15$ are space forms, and in particular Einstein.

\subsection{$\dim \ggo=5$}

\subsubsection{$\dim \ggo = 5, \dim \ngo = 1$}\label{m5g5n1}$\,$

Here, $\ug = \hg$ is a reductive Lie algebra isomorphic to $\RR^4,  \RR \oplus \slg_2(\RR)$ or $\RR \oplus \sug(2)$. Moreover, $\ug \neq \RR^4$ by Lemma \ref{lemadimn}, \eqref{lemadimn-zn2}.

Let $\{Y_1,Y_2,Y_3,Y_4\}$ be an orthonormal basis of $\ug =\RR \oplus \vg$, $\vg =\sug(2)$ or $\slg_2(\RR)$, such that $\{Y_1,Y_2,Y_3\}$ is a basis for $\vg$ and the Lie bracket of $\ug$ restricted to $\vg$ is given by (see~ \cite{Mln}):
\begin{equation}\label{corcabd}
[Y_1,Y_2]=d Y_3, \quad [Y_2,Y_3]=a Y_1, \quad [Y_3,Y_1]=b Y_2 \quad a,b,d \in \RR,\quad a,b,d \neq 0.
\end{equation}
From the Jacobi condition, $\ad_{\ug}(Y_4)$ is a derivation of $\vg,$ and so we have that $\ad_{\ug}(Y_4)=e \ad_{\ug}(Y_1)+f \ad_{\ug}(Y_2)+ g \ad_{\ug}(Y_3)$, for some $e,f,g\in \RR$. Thus we can compute the Ricci operator $\ricci_{\ug}$ with respect to this basis, in terms of $a,b,d,e,f$ and $g$. It is worth noticing that some of the off-diagonal entries of the matrix of $\ricci_\ug$ with respect to the basis $\{Y_1,Y_2,Y_3,Y_4 \}$ are
\begin{equation}\label{Riccioffdiag}
-\tfrac12 (b-~d)^2 e, -\tfrac12 (a-~d)^2 f, -\tfrac12 (a-b)^2 g.
\end{equation}
Also, since $\dim \ngo = 1$ we have that $\theta|_\vg = 0$. Using Theorem \ref{solisemal}, (ii) we obtain that
\[
\ricci_{\ug}=\left[
               \begin{smallmatrix}
                 c &  &  &  \\
                  & c &  &  \\
                  &  & c &  \\
                  &  &  & c+ \ast
               \end{smallmatrix}
             \right].
\]
Therefore, all three entries in \eqref{Riccioffdiag} must vanish. We will show that this cannot happen if $c<0$. Indeed, assume first that $a,b,d$ are pairwise different. Then we must have that $e=f=g=0$, and thus $\ricci_\ug|_\vg$ is actually the Ricci operator of a left-invariant metric on $\widetilde{\Sl_2(\RR)}$ or $\SU(2)$; this is a contradiction because these groups do not admit left-invariant Einstein metrics with $c<0$. On the other hand, if two among $a,b,d$ are equal, say $a=b$, then either $e=f=0$ or $b=d$. In any case, we obtain $c = (\ricci_\ug)_{33} = \tfrac{d^2}{2} > 0$, which is a contradiction.
%
%By using Maple (si???? ............ ), we have that the unique solution such that $(\ricci_{\ug})_{11} = (\ricci_{\ug})_{22} = (\ricci_{\ug})_{33}$ and $(\ricci_{\ug})_{ij}=0,$ for all $i \neq j,$ satisfies $a=b=d$, and in this case $c=\frac{1}{2}a^2 >0$, a contradiction.

\subsubsection{$\dim \ggo = 5, \dim \ngo = 2$}$\,$

We have that $\ug$ must be isomorphic to $\RR^3, \sug(2)$ or $\slg_2(\RR)$. If $\ug \simeq \RR^3,$ then $G$ is solvable and $(G,g)$ is a solvsoliton. On the other hand, from Corollaries \ref{Unocomp} and \ref{unosl2} we have that $\hg$ cannot be isomorphic neither to $\sug(2)$ nor to $\slg_2(\RR).$
%From Lemma \ref{lemadimn}, \eqref{lemadimn-zn2} $\ug \neq \RR^3$, and from Corollary \ref{Unocomp} $\ug \neq \sug(2)$, so $\ug = \slg_2(\RR)$. If $\theta=0$ we are in a product case, so we will focus in the case where $\theta$ is faithful. It is not hard to show by a straightforward calculation that to ensure that the compatibility condition (iv) from Theorem \ref{solisemal} holds, there must be an orthonormal basis of $(\ngo = \RR^2,\ip|_\ngo)$ such that $\theta$ is the standard representation of $\slg_2(\RR)$ by $2\times 2$ traceless matrices. Using a Milnor basis $\{e_1,e_2,e_3 \}$ as in the previous case, with $a<0$ and $b,d>0$, we obtain the operators $\ricci_\ug$ and $C_\theta$ diagonalized, and they are given by
%\[
%\ricci_\ug = \left[ \begin{smallmatrix} \tfrac12 (a^2 - (b-d)^2) &  & \\  & \tfrac12 (b^2 - (a-d)^2) &  \\ &  & \tfrac12 (d^2-(a-b)^2)  \end{smallmatrix}\right], \qquad
%C_\theta = \left[ \begin{smallmatrix} 0  &  &\\  &  -\tfrac{ad}{2} &  \\ &  & -\tfrac{ab}{2}  \end{smallmatrix}\right].
%\]
%We observe that $(\ricci_\ug)_{22}-(C_\theta)_{22} - (\ricci_{\ug})_{33} + (C_\theta)_{33} = \tfrac12 (b-d)(-3a + 2b + 2d)$, and since it must vanish and $-3a + 2b + 2d > 0$, we have $b=d$. But then looking at the $1-1$ entry, we see that $c = \tfrac12 a^2 > 0$, which is a contradiction. Therefore, there are no solitons in this case.

%%Terminar!!!  ($\ug = \slg_2(\RR)$ ??? )..........

\subsubsection{$\dim \ggo = 5, \dim \ngo \geq 3$}$\,$

It follows from Lemma \ref{lemadimn}, \eqref{lemadimn-u2} that in this case $(G,g)$ is a solvsoliton.

\subsection{$\dim \ggo=6$}$\,$

By \eqref{B} we will only consider the cases $\dim \ngo = 1,2,3$.

\subsubsection{$\dim \ggo = 6, \dim \ngo = 1$}$\,$

This case is empty, since by Lemma \ref{lemadimn}, \eqref{lemadimn-zn2} we must have $\dim \zg(\ug) \leq 1$ and any  $5$-dimensional reductive Lie algebra $\ug$ has $\dim \zg(\ug) \geq 2$.

\subsubsection{$\dim \ggo = 6, \dim \ngo = 2$}$\,$

In this case, $(U/K,g)$ is a $3$-dimensional homogeneous space with $U$ reductive and $\dim U = 4$. According to the table in \cite[pp.~58]{BB4}, there are four possibilities for $U/K$. The case $U = \U(2)$ can be omitted by Corollary \ref{Unocomp}. If $U = \RR \times \SO(3), K = \SO(2)$ then $g$ is a product of metrics of constant curvature. In particular, $\ricci_{U/K}$ has two positive eigenvalues. But $\Lie(\SO(3)) = \sug(2)$ and then $\theta|_{\sug(2)} = 0$ since $\dim \ngo = 2$. Hence $\dim \ker(C_\theta) \geq 2$, and this together with Theorem~ \ref{solisemal},~ (ii) contradict the fact that $c<0$.

Therefore, the only possibility is that $\ug = \RR\oplus \slg_2(\RR)$. Following \cite[Example 2.8]{homRF}, we take a basis $\{Z,Y_1,Y_2,Y_3 \}$ of $\ug$ such that $\{Y_1,Y_2,Y_3 \}$ is an orthonormal basis for $\hg$ and the Lie bracket is given by
\begin{align*}
&[Z,Y_2] = Y_3, & [Y_2,Y_3] = a Y_1 + b Z, & &[Y_1,Y_3] = -d Y_2,\\
&[Z,Y_3] = -Y_2, & & &[Y_1, Y_2] = d Y_3,
\end{align*}
with $a,b,d \in \RR$, $ad+b<0$. Then the Ricci operator for $U/K$ with respect to this basis is given by
\begin{equation}\label{Ricabd}
\ricci_{U/K} = \left[ \begin{array}{ccc} \tfrac12 a^2 & & \\ &-\tfrac12 a^2 + b + ad& \\ & & -\tfrac12 a^2  + b + ad \end{array}\right].
\end{equation}
Observe that the center of $\ug$ is generated by $Y_1 - d Z$. Since $\ngo$ is the nilradical of $\ggo$, $\theta|_{\zg(\ug)}$ is faithful. There are only two possibilities for $\theta$: either it is faithful, or $\theta|_{\slg_2(\RR)} = 0 $. Fix an orthonormal basis for $(\ngo, \ip|_\ngo)$, so that $\End(\ngo)$ may be identified with $2\times 2$ matrices.

If $\theta|_{\slg_2(\RR)} = 0$, then using that $\slg_2(\RR) = [\ug,\ug]$ we have that $a\theta(Y_1) + b \theta(Z) = 0$. In particular, if $a\neq 0$ then $\theta(Y_1)$ is skew-symmetric, so $C_\theta Y_1 = 0$. Using \eqref{Ricabd} we see that in order for condition (ii) of Theorem \ref{solisemal} to hold we must have $c = \tfrac12 a^2 \geq 0$, and this is a contradiction. On the other hand, if $a=0$, then $\{Z,Y_2,Y_3 \}$ is a basis for $\slg_2(\RR)$ and in particular $\theta(Z) = 0$. Then the center $\zg(\ug) = \RR (Y_1-d Z)$ is $B$-orthogonal to $\kg$, so $\zg(\ug)\subseteq \hg$ and $d=0$. Conditions (ii) and (iv) from Theorem \ref{solisemal} are satisfied if and only if $\theta(Y_1)$ is a normal operator and $\tr S(\theta(Y_1))^2 = -b$. If this is the case, we obtain an expanding algebraic soliton with $U/K = \RR \times \Sl_2(\RR) / \SO(2)$, and the metric $g$ is a Riemannian product
\[
\Sl_2(\RR)/\SO(2) \times S,
\]
where $\Sl_2(\RR)/\SO(2) \simeq \RR H^2$ is the hyperbolic $2$-space and $S$ is a $3$-dimensional solvsoliton with Lie algebra given by $\RR Y_1 \ltimes_\theta \RR^2$.

%, whose metric is clearly isometric to a product of a constant curvature metric on $\Sl_2(\RR)/\SO(2)$ and a $3$-dimensional solvsoliton metric on $S = \RR \ltimes \RR^2$, the action of $\RR$ on $\RR^2$ being given at the Lie algebra level by the normal operator $\theta(Y_1)$ ........ chequear toda la ultima oracion!!

Now assume that $\theta$ is faithful. We have that
\[
\theta(Y_1 - d Z) = e \left[ \begin{smallmatrix} 1 & 0\\0 & 1 \end{smallmatrix} \right],
%\theta(Z) = \tfrac12 \left[ \begin{smallmatrix} 0 & -1\\ 1& 0 \end{smallmatrix} \right], \qquad
%\theta(Y_2) = \tfrac{\sqrt{-b}}{2}  \left[ \begin{smallmatrix} 1 & 0\\ 0& -1 \end{smallmatrix} \right], \qquad
%\theta(Y_3) = \tfrac{\sqrt{-b}}{2} \left[ \begin{smallmatrix} 0 & 1\\ 1& 0 \end{smallmatrix} \right],
\]
for some $e\neq 0$. This implies first that $Y_1-dZ$ is Killing-orthogonal to $Z$, thus $d=0$ by definition of $\pg$. On the other hand, applying $\theta$ and taking traces on $[Y_2,Y_3] = a Y_1 + b Z$ yields $a \tr \theta(Y_1) = 0$, hence $a=0$. To sum up, we obtained that $\zg(\ug) = \RR Y_1 \subseteq \hg$, and that $\{Z,Y_2,Y_3 \}$ is a basis for the ideal $\slg_2(\RR)$. In order to better visualize the results in this case, we proceed by varying inner products instead of Lie brackets. So we set $a=d=0$, $b=-4$, and since any representation of $\RR \oplus \slg_2(\RR)$ is equivalent to the standard one, assume that
\[
\theta(Z) = \tfrac12 \left[ \begin{smallmatrix} 0 & -1\\ 1& 0 \end{smallmatrix} \right], \qquad
\theta(Y_1) = \left[ \begin{smallmatrix} 1 & 0\\0 & 1 \end{smallmatrix} \right], \qquad
\theta(Y_2) = \left[ \begin{smallmatrix} 1 & 0\\ 0& -1 \end{smallmatrix} \right], \qquad
\theta(Y_3) = \left[ \begin{smallmatrix} 0 & 1\\ 1& 0 \end{smallmatrix} \right].
\]
The isotropy representation has a decomposition into irreducible factors of the form $\pg = \pg_1 \oplus \pg_2 \oplus \pg_3$, where $Y_1\in \pg_1$, $Y_2,Y_3 \in \pg_2$, $\ngo = \pg_3$. Fix an inner product $\ip_0$ that makes the basis $\{Y_1,Y_2,Y_3,X_1,X_2 \}$ orthonormal, and then any $\ad(\kg)$-invariant inner product on $\pg$ that makes $\hg \perp \ngo$ is of the form
\[
\ip_{\alpha,\beta,\gamma} = \alpha \cdot \ip_0|_{\pg_1} + \beta  \cdot \ip_0|_{\pg_2} + \gamma \cdot \ip_0|_{\pg_3}, \qquad \alpha,\beta,\gamma > 0.
\]
Conditions (i), (iii) and (iv) from Theorem \ref{solisemal} are clearly satisfied for these inner products. Regarding condition (ii), we calculate the Ricci operator of $U/K$ with the metric induced by $\ip_{\alpha,\beta,\gamma}$, and $C_\theta$, with respect to the orthonormal basis $\{\tfrac{Y_1}{\sqrt \alpha},\tfrac{Y_2}{\sqrt \beta},\tfrac{Y_3}{\sqrt \beta} \}$, to obtain:
\[
\ricci_{U/K,\alpha,\beta} =  \left[ \begin{array}{ccc}  0 & & \\ &  -\tfrac{4}{\beta}& \\ & &  -\tfrac{4}{\beta} \end{array}\right], \qquad
C_\theta =  \left[ \begin{array}{ccc}  \tfrac{2}{\alpha} & & \\ & \tfrac{2}{\beta}& \\ & &  \tfrac{2}{\beta} \end{array}\right].
\]
Thus, $c=-2/\alpha$, and in order for (ii) to hold we must have $\beta = 3 \alpha$. We see that, up to a scalar multiple, these algebraic solitons are all isometric to each other, and they are actually Einstein.

Arguing as in the case \ref{M4G5N2} we see that in both cases ($\theta|_{\slg_2(\RR)} = 0$ or $\theta$ faithful), all expanding algebraic solitons that we obtained are isometric to solvsolitons.

\subsubsection{$\dim \ggo = 6, \dim \ngo = 3$}\label{m5g6n3}$\,$

Assume first that $\ngo$ is abelian. Observe that $(U/K, g)$ is a $2$-dimensional space-form, with $\ug$ reductive and non-compact, thus up to a cover we have $U = \Sl_2(\RR)$, $K = \SO(2)$, $\ug = \slg_2(\RR)$, and $\kg = \Lie(\SO(2)) = \RR Z$, where $Z$ is a $2\times 2$ skew-symmetric matrix in $\slg_2(\RR)$. As in the previous case, we proceed by varying inner products. Thus, let us disregard the inner product fixed on $\pg$, and consider only the Lie-theoretical information. There are two non-trivial, non-equivalent representations of $\slg_2(\RR)$ on $\RR^3$: one is given by the adjoint representation, and we will call it $\theta_{\ad}$; the other is given by a direct sum of the standard $2\times 2$ matrix representation of $\slg_2(\RR)$ and a trivial one-dimensional representation, and we will call it~ $\theta_{12}$. In order to determine all possible $G$-invariant metrics which are expanding algebraic solitons we seek for $\ad(\kg)$-invariant inner products on $\pg$ satisfying conditions (i)-(iv) from Theorem \ref{solisemal}. In both cases, the linear isotropy representation has a decomposition into irreducible factors of the form $\pg = \pg_1 \oplus \pg_2 \oplus \pg_3$, where $\pg_1 = \hg$, $\pg_2 \oplus \pg_3 = \ngo$ and $\dim \pg_2 = 2,\dim \pg_3 = 1$.

In the case of $\theta_{\ad}$, $\pg_1$ is equivalent to $\pg_2$. However, Theorem \ref{solisemal} implies that in order to give rise to an algebraic soliton, an inner product $\ip$ must satisfy $\langle \pg_1, \pg_2 \oplus \pg_3\rangle = 0$.  Fix a basis $\{Y_1, Y_2\}$ of $\hg$ such that $[Z,Y_1] = -2 Y_2, [Z,Y_2] = 2Y_1, [Y_1,Y_2] = 2 Z$, and consider the basis $\{X_1,X_2,X_3\}$ of $\ngo$ such that, with respect to this basis,
\[
\theta_{\ad} (Z) = \left[\begin{smallmatrix} 0 & 2 & 0 \\ -2 & 0 & 0 \\ 0 & 0 & 0 \end{smallmatrix}\right], \quad
\theta_{\ad} (Y_1) = \left[\begin{smallmatrix} 0 & 0 & 0 \\ 0 & 0 & 2 \\ 0 & 2 & 0 \end{smallmatrix}\right], \quad
\theta_{\ad} (Y_2) = \left[\begin{smallmatrix} 0 & 0 & -2 \\ 0 & 0 & 0 \\ -2 & 0 & 0 \end{smallmatrix}\right].
\]
Fix also an inner product $\ip_0$ on $\pg$ that makes $\{Y_1,Y_2,X_1,X_2,X_3 \}$ orthonormal; it is easy to see that $\ip_0$ satisfies (i)-(iv) from Theorem \ref{solisemal}. All $\ad(\kg)$-invariant inner products on $\pg$ that make $\pg_1\perp \pg_2$ are given by
\[
\ip_{\alpha,\beta,\gamma} = \alpha \cdot \ip_0|_{\pg_1}  + \beta \cdot \ip_0|_{\pg_2}  + \gamma \cdot \ip_0|_{\pg_3}, \qquad \alpha,\beta,\gamma>0,
\]
and for $\ip_{\alpha,\beta,\gamma}$ in order to satisfy condition (iv) we must have that $\beta=\gamma$. It is now easy to check that any inner product of the form $\ip_{\alpha,\beta,\beta}$ satisfies (i)-(iv) and thus gives rise to an expanding algebraic soliton.

For $\theta_{12}$, all three summands $\pg_1, \pg_2, \pg_3$ are mutually inequivalent representations. As in the previous case, fix an inner product $\ip_0$ that makes the basis $\{Y_1,Y_2,X_1,X_2,X_3 \}$ orthonormal, where $\{X_1,X_2,X_3\}$ is such that
 %(here, $\{ X_1,X_2 \}$ is a basis of $\pg_2$ such that $\theta_{12}(Z)X_1 = - X_2, \theta_{12}(Z)X_2 = X_1$, and $X_3\in \pg_3$).
\[
\theta_{12} (Z) = \left[\begin{smallmatrix} 0 & 1 & 0 \\ -1 & 0 & 0 \\ 0 & 0 & 0 \end{smallmatrix}\right], \quad
\theta_{12} (Y_1) = \left[\begin{smallmatrix} 1 & 0 & 0 \\ 0 & -1 & 0 \\ 0 & 0 & 0 \end{smallmatrix}\right], \quad
\theta_{12} (Y_2) = \left[\begin{smallmatrix} 0 & 1 & 0 \\ 1 & 0 & 0 \\ 0 & 0 & 0 \end{smallmatrix}\right].
\]
All $\ad(\kg)$-invariant inner products on $\pg$ are given by
\[
\ip_{\alpha,\beta,\gamma} = \alpha \cdot \ip_0|_{\pg_1}  + \beta \cdot \ip_0|_{\pg_2}  + \gamma \cdot \ip_0|_{\pg_3}, \qquad \alpha,\beta,\gamma>0,
\]
and an easy calculation shows that all of them satisfy (i)-(iv) from Theorem \ref{solisemal}.

Regarding $\ngo$, the other possibility is that $\ngo = \hg_3$, the $3$-dimensional Heisenberg Lie algebra. Up to equivalence, there is only one non-trivial representation $\theta: \slg_2(\RR)\to \Der(\hg_3)$, and is given precisely by $\theta_{12}$. The above argument and a straightforward calculation show that an $\ad(\kg)$-invariant inner product $\ip_{\alpha,\beta,\gamma}$ on $\pg$ gives rise to an algebraic soliton if and only if $\alpha = 4 \beta^2 /\gamma $ (the only condition from Theorem \ref{solisemal} that is not satisfied for all $\alpha, \beta, \gamma$ is (ii) ). 

On the other hand, if $\ngo = \hg_3$ and $\theta=0$ we obtain a Riemannian product of the hyperbolic $2$-space $\Sl_2(\RR)/\SO(2) = \RR H^2$ and the $3$-dimensional nilsoliton on the Heisenberg Lie group $H_3$, and this product is an algebraic soliton if and only if the comsmological constants of both spaces coincide.

Finally, arguing as in the case \ref{M4G5N2} we see that all algebraic solitons that we obtained are isometric to solvsolitons. In the case of $\theta_{\ad}$, they are all homothetic (i.e.~isometric up to scaling) to a fixed $5$-dimensional solvsoliton. For $\theta_{12}$, $\ngo = \RR^3$, they are all homothetic to the product of $\RR$ and the $4$-dimensional solvsoliton corresponding to the case \ref{M4G5N2}. And for $\theta_{12}$, $\ngo = \hg_3$, they are all homothetic to another $5$-dimensional solvsoliton.

%Lo de la dimension de los nuevos nilradicales esta mal!!! sacar! ...............

%

\subsection{$\dim \ggo=7$} $\,$

By \eqref{B} we need only to consider the cases $\dim \ngo = 1,2$.

\subsubsection{$\dim \ggo = 7, \dim \ngo =1$}$\,$

By Lemma \ref{lemadimn}, \eqref{lemadimn-zn2} we have that $\dim \zg(\ug) \leq 1$. Since $\ug$ is reductive and $6$-dimensional, it must be semisimple. Using Lemma \ref{lemadimn}, \eqref{lemadimn-n1uss} we obtain $\theta = 0$, hences this is a product case.

\subsubsection{$\dim \ggo = 7, \dim \ngo =2$}$\,$

This case is empty because $U/K$ is almost effective and $3$-dimensional, with $\dim U = 5$, and this contradicts Proposition \ref{gap}.

\subsection{$\dim \ggo=8$}$\,$

As in the previous case, we need only to consider $\dim \ngo = 1,2$.

\subsubsection{$\dim \ggo = 8, \dim \ngo =1$}\label{m5g8n1}$\,$

Assume that $\ggo$ is non-unimodular and $\theta\neq 0$, otherwise by Lemma \ref{lemadimn}, \eqref{lemadimn-n1uss} we are in a product case. Then, by using Proposition \ref{unimod} we get a simply connected expanding algebraic soliton $(G_1/K, g)$ of dimension 4, with $G_1$ a normal subgroup of $G$ of codimension one. It is not Einstein, since $G_1$ is unimodular and not semisimple, thus it must be one of the spaces listed in Table \ref{ta4}. Thus, $\ggo$ contains an ideal isomorphic to $\RR \oplus \slg_2(\CC)$. On the other hand, $U/K$ is an almost effective homogeneous space of dimension 4, and comparing with the table in \cite[pp.~45]{BB4} we conclude that the only possibility is $U = \RR \times \SO_0(3,1)$ (up to a cover), $K = \SO(3)$, and the metric restricted to $U/K$ is a product of $\RR$ and a metric of constant negative curvature on $\SO_0(3,1)/\SO(3) = \RR H^3$. Let $c<0$ be the Einstein constant of the metric on $\SO_0(3,1)/\SO(3)$, and let $\slg_2(\CC) = \Lie(\SO_0(3,1)) \subseteq \ug$. It is clear that the only possibility for $\theta$ is that
\[
\theta|_{\slg_2(\CC)} = 0,\qquad \theta(W) = -c,
\]
where $W\in \zg(\ug)\subseteq \hg$ is a unitary vector. In this way, we obtain an algebraic soliton on the homogeneous space $\left(\SO_0(3,1) \times Z_2\right)/\SO(3)$. A quick inspection of its Ricci operator shows that this is actually an Einstein metric.

\subsubsection{$\dim \ggo = 8, \dim \ngo =2$}$\,$

Since we have equality in \eqref{B}, $(U/K, g|_{U/K})$ is a space form. By Corollary \ref{Unocomp} it cannot have positive curvature, thus up to a cover the Lie group $U$ is $\SO(3)\ltimes \RR^3$ or $\SO_0(3,1)$. But $\SO(3)\ltimes \RR^3$ is not reductive, so the only remaining case is $U = \SO_0(3,1)$ (up to a cover). In this case, $\ug \simeq \slg_2(\CC)$, a $6$-dimensional simple Lie algebra, thus $\theta$ is either faithfull or trivial. Since $\dim \ngo = 2$, it must be trivial, and we are in a product case.

\subsection{$9\leq \dim \ggo \leq 11$}$\,$

The cases with $\dim \ngo \geq 2$ may be disregarded because of \eqref{B}, so let us assume that $\dim \ngo = 1$. If $\ggo$ is unimodular then by Lemma \ref{lemadimn}, \eqref{lemadimn-n1uss} we are in a product case. On the other hand, if $\ggo$ is non-unimodular then by using Proposition \ref{unimod} and arguing as in Section \ref{m5g8n1} we get that $G_1$ must be one of the groups in Table \ref{ta4}. This in particular implies that $7\geq \dim G_1 = \dim G - 1$, which is a contradiction.

\section{Proof of Theorem \ref{main6}: $\dim G/K =6$}\label{dim6}

In this section we will show that, besides from the cases where $G$ is semisimple, there is no counterexample to the generalized Alekseevskii conjecture in dimension $6$. By Corollary \ref{cormain}, the generalized Alekseevskii's conjecture is valid up to dimension $5$. So, by using Proposition \ref{unimod} we may assume without loss of generality that $G$ is unimodular. The cases with $\dim \ngo = 0$ correspond to $G$ semisimple, and will be omitted. For $\dim \ngo = 1$, we are in a product case by Lemma \ref{lemadimn}, \eqref{lemadimn-n1uss}. The possible dimensions for $G$ are therefore $6\leq \dim G \leq 16$ and $\dim G = 21$. When $\dim G = 21$, $G/K$ is the hyperbolic space, which is diffeomorphic to $\RR^6$.

\subsection{$\dim \ggo = 6$}$\,$

This is the case of left-invariant metrics on $6$-dimensional Lie groups.

\subsubsection{$\dim \ggo =6, \dim \ngo = 2$}$\,$

We have that $\dim \ug = 4$ and thus it must be isomorphic to one of the following:
\[
\RR^4, \qquad \RR\oplus \slg_2(\RR), \qquad \RR\oplus \sug(2).
\]
If $\ug = \RR^4$ then $\ggo$ is solvable, so we may apply Lemma \ref{gsol} to conclude that $(G/K,g)$ is isometric to a solvsoliton. For the other cases, recall that since $\ngo$ is the nilradical we have that $\theta|_{\zg(\ug)}$ is faithful. Also, if $\theta: \ug \to \End(\ngo)$ is faithful, then $\ggo$ is not unimodular. Therefore, we must have $\ker \theta = \vg$, with $\vg$ a simple ideal. But then, arguing along the lines of Section \ref{m5g5n1}, we arrive at a contradiction.

\subsubsection{$\dim \ggo =6, \dim \ngo = 3$}$\,$

In this case, $\ug$ is one of the following: $\RR^3, \slg_2(\RR), \sug(2)$. But using Corollaries~ \ref{Unocomp} and \ref{unosl2} we have that $\ug \neq \sug(2)$ and $\ug \neq \slg_2(\RR)$, so $\ug = \RR^3$, $\ggo$ is solvable and $(G,g)$ is a solvsoliton.  %Finally, if $\ug = \slg_2(\RR)$, then as a differentiable manifold $G = \widetilde{\Sl_2(\RR)} \times N$, and since $N$ and $\widetilde{\Sl_2(\RR)}$ are diffeomorphic to $\RR^3$ ($N$ is a simply connected nilpotent Lie group) we have that $G$ is diffeomorphic to $\RR^6$.

\subsubsection{$\dim \ggo =6, \dim \ngo \geq 4$}$\,$

By Lemma \ref{lemadimn}, \eqref{lemadimn-u2} we have that $(G,g)$  is a solvsoliton.

\subsection{$\dim \ggo > 6$}$\,$

In these cases, we will make strong use of Corollary \ref{U/Kae} and equation \eqref{B}. Also, to deal with the cases where $\ggo$ has smaller dimension, we will do as follows: given an expanding algebraic soliton $(G/K,g)$, we take a presentation of it such that $G$ is simply connected. We may lose effectiveness, but we will still have that the action is almost-effective. Thanks to Remark \ref{Gsc} we know that $G/K = U/K \times N$ as differentiable manifolds, so in particular $U/K$ and $N$ are also simply connected. Since $N$ is nilpotent, it is diffeomorphic to a Euclidean space. Therefore, the theorem will be proved if we show that $U/K$ is diffeomorphic to a Euclidean space. In order to do so, we will rely on the fact that homogeneous spaces of dimension $\leq 4$ are classified (see \cite{BB4}), and looking at the tables of the classification we will have to rule out those cases that have a compact factor.

\subsubsection{$\dim \ggo >6, \dim \ngo = 2$}$\,$

By \eqref{B} we have that $\dim \ggo \leq 12$, and also $\dim \ggo \neq 11$ since there are no $4$-dimensional homogeneous spaces with a $9$-dimensional transitive group.

If $\dim \ggo = 12$, then $(U/K, g_{U/K})$ is a simply connected space form. Since $U$ is non-compact and reductive, we are in the case of the hyperbolic space, which is diffeomorphic to $\RR^4$.

If $8 \leq \dim \ggo \leq 10$, we look closely at the table in \cite[pp.~45]{BB4}. The only cases with $\ug$ non-compact and reductive for which $U/K$ is not diffeomorphic to a Euclidean space are $\ug = \RR \oplus \sog(4)$ and $\ug = \sog(3) \oplus \slg_2(\RR)$. For $\ug = \RR \oplus \sog(4)$, since $\dim \ngo = 2$ we have that $\theta$ must be trivial on the simple ideal $\sog(4)$. But the metric is a Riemannian product of a flat factor $\RR$ and a constant (positive) curvature metric on $\SO(4)/\SO(3) \simeq S^3$. We arrive at a contradiction by looking at the Ricci operator and using Theorem \ref{solisemal}, (ii). For the case $\ug = \sog(3) \oplus \slg_2(\RR)$ we also obtain a contradiction by following a similar argument, this time focused on the ideal $\sog(3)$.

If $\dim \ggo = 7$, we now look closely at the tabe in \cite[pp.~44]{BB4}. There are three possibilities for $U/K$ that are not diffeomorphic to a Euclidean space: $\RR^2 \times \SO(3) / \SO(2)$, $Z^2 \times \SO(3)/\SO(2)$ and $\RR \times \U(2) / \U(1)$. The case $Z^2 \times \SO(3)/\SO(2)$ is not possible because $Z^2 \times \SO(3)$ is not reductive. For $\RR^2 \times \SO(3) / \SO(2)$, the metric is a product of constant curvature metrics. We see that $\theta|_{\sog(3)}$ must be trivial because $\dim \ngo = 2$, and thus we arrive at a contradiction by looking at the Ricci operator and using Theorem \ref{solisemal}, (ii), as in the cases from the previous paragraph. Finally, if $U/K = \RR \times \U(2) / \U(1)$, the metric is also a product, but the metric on $\U(2)/\U(1)$ is a Berger metric on $S^3$. Considering the restriction of $\theta$ to the ideal $\ug(2)$ of $\ug$, we see that it must act by skew-symmetric endomorphisms on $\ngo$. Indeed, we have that $\ug(2) = \zg(\ug(2)) \oplus \sug(2)$, and $\theta|_{\sug(2)} = 0$ since $\dim \ngo = 2$. But notice that we cannot have $\kg \subseteq \sug(2)$ because otherwise the homogeneous space would be $\RR^2 \times \SO(3)/\SO(2)$, so we may write a nonzero $Z\in \kg$ as $Z = Y_0 + Y_1$, with $Y_0\in\zg(\ug(2)), Y_1\in \sug(2)$ and $Y_0\neq 0$. In this way, we see that $\theta(Y_0)$ must be skew-symmetric, and this implies that $\theta(Y)$ is skew-symmetric for all $Y\in \ug(2)$. Using condition (ii) from Theorem \ref{solisemal} we arrive at a contradiction, because $C_\theta$ is $0$ on $\ug(2)$, and the Ricci operator of a Berger metric of $S^3$ cannot be negative definite.

\subsubsection{$\dim \ggo >6, \dim \ngo = 3$}$\,$

From \eqref{B} we have that $\dim \ggo \leq 9$ in this case, and $\dim \ggo \neq 8$ since a $3$-dimensional homogeneous space does not admit a $5$-dimensional transitive and effective group. As in the previous section, if $\dim \ggo = 9$, then $(U/K, g_{U/K})$ must be the $3$-dimensional hyperbolic space, which is diffeomorphic to $\RR^3$.

If $\dim \ggo = 7$, we look at the table in \cite[pp.~58]{BB4} to see the possibilities for $U/K$. There are two cases which are not diffeomorphic to $\RR^3$:
\begin{align*}
\left(\RR\times \SO(3)\right)&/ \SO(2) \simeq \RR \times S^2, \qquad \hbox{product of constant curvature metrics,}\\
\U(2)&/U(1) \simeq S^3, \qquad \hbox{Berger metric.}
\end{align*}
Since $\U(2)$ is compact, the second case is not possible because of Corollary \ref{Unocomp}. Assume now that $\ug = \RR \oplus \sog(3), \kg = \RR \subseteq \sog(3)$. By Theorem \ref{solisemal}, (ii) we must have that $\theta|_{\sog(3)}$ is faithful, otherwise $c$ would be positive. Since $\theta|_{\zg(\ug)}$ is also faithful, we conclude that $\theta: \ug \to \End(\RR^3)$ is faithful. This in turn implies that $\theta(Y_0) = e Id$, where $Y_0 \in \zg(\ug)$. Consider the restriction of $\theta$ to the ideal $\sog(3)$: it also satisfies the compatibility condition (iv) from Theorem \ref{solisemal}. By arguing as in the proof of Proposition \ref{cartan}, we can prove that $\theta(Y)^t = -\theta(Y)$ for all $Y\in \sog(3)$. This contradicts condition (ii) from Theorem \ref{solisemal}, since $c<0$ and the metric on $\SO(3)/\SO(2)$ has positive curvature.

We conclude that $U/K$ must be diffeomorphic to $\RR^3$.

\subsubsection{$\dim \ggo >6, \dim \ngo = 4$}$\,$

Using \eqref{B}, the only possibility in this case is that $\dim \ggo = 7$. But then $U/K$ is a $2$-dimensional space form, so it must be the $2$-dimensional hyperbolic space, and this is diffeomorphic to $\RR^2$.

\subsubsection{$\dim \ggo >6, \dim \ngo \geq 5$}$\,$

In this case $G$ is solvable and by Lemma \ref{lemadimn}, \eqref{lemadimn-u2} $(G/K,g)$ is isometric to a solvsoliton; in particular it is diffeomorphic to $\RR^6$.

\section{Appendix - by Jorge Lauret}\label{app}

%\begin{proposition}\cite{Jorge}\label{Unocomp}
%Let $(G/K,g)$ be an expanding algebraic soliton, consider the decomposition $\gkhn$ as in $\eqref{descdeg}$ and assume that $\hg\neq 0$. Then, $U$ cannot be compact.
%\end{proposition}

Let $(G/K,g)$ be an expanding algebraic soliton with metric reductive decomposition $\ggo=\kg\oplus\hg\oplus\ngo$ as in Theorem \ref{solisemal}. Let $U$ denote the connected Lie subgroup of $G$ with Lie algebra $\ug=\kg\oplus\hg$.  Recall that we
can assume that $K$ is a compact Lie subgroup of $U$, and thus $K$ is contained in
some maximal compact subgroup $K_{max}$ of $U$ which is connected (see \cite{Iws,Hch}) with Lie algebra $\kg_{max}$.

Let us assume that either $U$ is compact (i.e. $\kg_{max}=\ug$) or $\ug$ is semisimple and there is a direct sum decomposition
$$
\hg=\hg^{-}\oplus\hg^{+},
$$
such that $\kg_{max}=\kg\oplus\hg^{-}$ and $\ug=\kg_{max}\oplus\hg^+$ is a Cartan decomposition of $\ug$, that is,
$$
[\kg_{max},\hg^+]\subset\hg^+, \qquad [\hg^+,\hg^+]\subset\kg_{max},
$$
and the Killing form of $\ug$ is negative definite on $\kg_{max}$ and positive definite on $\hg^+$.  It is worth mentioning that $\hg^-$ and $\hg^+$ are not necessarily orthogonal with respect to $\ip$.

\begin{proposition}\label{cartan}
If $\la\hg^-,\hg^+\ra=0$, then
$$
\theta(Y)^t=-\theta(Y) \quad\forall Y\in\hg^-, \qquad\quad \theta(Y)^t=\theta(Y)
\quad\forall Y\in\hg^+.
$$
\end{proposition}

\begin{proof}
We need some results on geometric invariant theory concerning moment maps for real representations of real reductive Lie groups (see e.g. \cite[Appendix]{cruzchica} for more information).  Given two inner product vector spaces $(\hg,\ip)$ and $(\ngo,\ip)$ we
consider the linear action of $\Gl(\ngo)$ on the vector space $W:=\End(\hg,\glg(\ngo))$
given by
\begin{equation}\label{actionw}
(h\cdot\theta)(Y) = h\theta(Y)h^{-1}, \qquad \forall \; Y\in\hg,\; h\in\Gl(\ngo),\; \theta\in W.
\end{equation}
The corresponding $\glg(\ngo)$-representation is then given by
$$
\pi(A)(\theta)= [A,\theta(\cdot)], \qquad\forall \; A\in\glg(\ngo),\;
\theta\in W.
$$
If $\{ Y_i\}$ is an orthonormal basis of $\hg$, then a canonical
$\Or(\ngo)$-invariant inner product on $\End(\hg,\glg(\ngo))$ is defined by
$$
\la\theta,\theta'\ra= \sum\la\theta(Y_i),\theta'(Y_i)\ra =
\sum\tr{\theta(Y_i)\theta'(Y_i)^t}.
$$
As a Cartan decomposition we take $\glg(\ngo)=\sog(\ngo)\oplus\sym(\ngo)$, thus the
moment map $m:W\longrightarrow\sym(\ngo)$ for the action
(\ref{actionw}) is given by
\begin{equation}\label{mmw}
m(\theta)=\tfrac{1}{|\theta|^2}\sum[\theta(Y_i),\theta(Y_i)^t],
\end{equation}
where $\{ Y_i\}$ can be actually any orthonormal basis of $\hg$.  Indeed, for any $A\in\sym(\ngo)$, we have that
$$
\begin{array}{rcl}
\la m(\theta),A\ra &=& \tfrac{1}{|\theta|^2}\la\pi(A)\theta,\theta\ra = \tfrac{1}{|\theta|^2}\sum\la[A,\theta(Y_i)],\theta(Y_i)\ra \\ \\
&=& \tfrac{1}{|\theta|^2}\sum\la A,[\theta(Y_i),\theta(Y_i)^t]\ra.
\end{array}
$$

We now apply this to $\hg$ and $\ngo$ coming from the reductive decomposition $\ggo=\kg\oplus\hg\oplus\ngo$ as above.  It follows from Theorem \ref{solisemal}, (iv) that
$$
m(\theta)=\tfrac{1}{|\theta|^2}\sum [\theta(Y_i),(\theta(Y_i))^t]=0,
$$
thus $\theta$ is a minimal vector. On the other hand, since $\ug=\kg_{max}\oplus\hg^+$ is a Cartan decomposition, there exists an inner product $\ip_1$ on $\ngo$ such that $\theta(\kg_{max})$ and
$\theta(\hg^+)$ are respectively contained in the space of skew-symmetric and
symmetric maps with respect to $\ip_1$ (see e.g. \cite[Lemma 3]{Dtt} or \cite{Dnl}).  We note that if $U$ is compact (and possibly non-semisimple) then the above assertion also holds.  This is equivalent to the existence of an
element $h\in\Gl(\ngo)$ such that
$$
\begin{array}{l}
(h\theta(Y)h^{-1})^t=-h\theta(Y)h^{-1}, \qquad\forall Y\in\hg^-, \\ \\
(h\theta(Y)h^{-1})^t=h\theta(Y)h^{-1}, \qquad\forall Y\in\hg^+.
\end{array}
$$
We therefore obtain that also $m(h\cdot\theta)=0$ and so according to \cite[Theorem 11.1, (iii)]{cruzchica},
there exists a $\vp\in\Or(\ngo)$ such that $h\cdot\theta=\vp\cdot\theta$.  It follows that for any $Y\in\hg^{\mp}$,
$$
\begin{array}{rcl}
\theta(Y)^t &=& \vp^{-1}(\vp\theta(Y)\vp^{-1})^t\vp = \vp^{-1}(\vp\cdot\theta)(Y)^t\vp
= \vp^{-1}(h\cdot\theta)(Y)^t\vp \\ \\
&=& \mp\vp^{-1}(h\cdot\theta)(Y)\vp = \mp\vp^{-1}(\vp\cdot\theta)(Y)\vp =\mp\theta(Y),
\end{array}
$$
as was to be shown.
\end{proof}

\begin{corollary}\label{Unocomp}
The Lie group $U$ can not be compact, unless $\hg=0$.
\end{corollary}

\begin{proof}
If $U$ is compact then $\hg=\hg^-$, or equivalently $\hg^+=0$, and thus by Proposition \ref{cartan} we obtain that $C_\theta=0$.  This implies that
$\Ricci_{U/K}=cI$, which is a contradiction, since the associated compact
homogeneous Riemannian manifold $U/K$ can never be Einstein of negative scalar
curvature by Bochner's Theorem (see \cite[Theorem 1.84]{Bss}).  It follows that $\hg=0$, concluding the proof.
\end{proof}

\begin{corollary}\label{unosl2}
If $\kg=0$ then the Lie algebra $\ug$ can not be isomorphic to $\slg_2(\RR)$.
\end{corollary}

\begin{proof}
Assume that $\ug\simeq\slg_2(\RR)$.  Since $\kg=0$ we have that $\ug=\hg$ and then there is always a Cartan decomposition $\ug=\hg^-\oplus\hg^+$ which is orthogonal.  Indeed, one can always find an orthonormal basis such that $[e_2, e_3] = ae_1$, $[e_3, e_1] = be_2$, $[e_1, e_2] = de_3$ with $a,b>0$, $d<0$ (see e.g. \cite{Mln}) and so $\hg^-=\RR e_3$, $\hg^+=\RR e_1+\RR e_2$ define a Cartan decomposition.  It follows from Proposition \ref{cartan} that $C_\theta e_3=0$ and $C_\theta|_{\hg^+}$ is a positive multiple of the Killing form.

We now use the formulas given in \cite[Example 2.7]{homRF} to obtain from $\Ricci_{\ug}=cI+C_\theta$ that
\[
 \unm (a^2-(b-d)^2) = c - 2\lambda bd, \qquad \unm (b^2-(a-d)^2) = c - 2\lambda ad, \qquad c=\unm(d^2-(a-b)^2),
\]
for some $\lambda>0$. By substracting the first two equations we obtain
\[
(a-b)(a+b-(2\lambda+1)d) = 0.
\]
Since $a+b-(2\lambda+1)d>0$, this implies that $a=b$, and putting this into the third equation yields $c = \unm d^2 > 0$, which is a contradiction.
%We now use the formulas given in \cite[Example 2.7]{homRF} to obtain from $\Ricci_{\ug}=cI+C_\theta$ that
%$$
%k=\unm(c^2-(a-b)^2), \qquad c^2-a^2=\lambda bc, \qquad c^2-b^2=\lambda ac,
%$$
%for some $\lambda>0$.  This implies that $c^2=a^2+ab+b^2$ and thus $k=\tfrac{3}{2}ab>0$, which is a contradiction.
\end{proof}

\bibliography{semialglow}
\bibliographystyle{amsalpha}

\end{document}